 \documentclass[letterpaper,11 pt]{article}
 
\usepackage{amsmath} 
\usepackage{amssymb}  
\usepackage{acronym,wrapfig,plain,mathrsfs,enumerate,relsize,color}
\usepackage[colorlinks=true,citecolor=blue]{hyperref}
\definecolor{blue}{RGB}{0,0,255} 

\allowdisplaybreaks
\usepackage{graphicx}
\usepackage{spverbatim}
\newtheorem{algorithm}{Algorithm}
\usepackage{algorithm}
\usepackage{amsmath}
\usepackage{multirow}
\usepackage{algorithmic}
\usepackage{subfig}
\usepackage{hyperref}       
\usepackage{url}            
\usepackage{booktabs}       
\usepackage{amsfonts}       
\usepackage{nicefrac}       
\usepackage{microtype}      
\usepackage{xcolor}         
\usepackage{graphicx}
\usepackage{appendix}
 
\usepackage{amsfonts,amsmath,fullpage,bbm}
\usepackage{amssymb,amsthm,multirow,verbatim}
\usepackage{acronym,wrapfig,plain,mathrsfs,enumerate,relsize,color}

\usepackage{amsmath,amsfonts}
\usepackage{algorithmic}
\usepackage{algorithm}
\usepackage{array}
\usepackage{witharrows}

\allowdisplaybreaks
\usepackage{spverbatim}
\usepackage{mdframed}
\usepackage{algorithm}
\usepackage{multirow}
\usepackage{tcolorbox}
\newtheorem{theorem}{Theorem}
\newtheorem{corollary}{Corollary}
\newtheorem{lemma}{Lemma}

\newtheorem{remark}{Remark} 
\newcommand{\argmin}{\mathop{\rm argmin}}

\newtheorem{assumption}{Assumption}
\newtheorem{definition}{Definition}

\newcommand{\za}[1]{{\color{black}#1}}

\newcommand{\zr}[1]{{\color{black}#1}}
\newcommand{\blue}[1]{{\color{black}#1}}

\newcommand{\eyh}[1]{{\textcolor{black}{#1}}}
\newcommand{\eyy}[1]{{\textcolor{black}{#1}}}
\newcommand{\ey}[1]{{\textcolor{black}{#1}}}
\newcommand{\zi}[1]{{\color{black}#1}}

\def\grad{\nabla}

\def\bD{\mathbf{D}}

\def\cC{\mathcal{C}}

\def\cE{\mathcal{E}}

\def\cG{\mathcal{G}}

\def\cO{\mathcal{O}}

\def\cS{\mathcal{S}}

\def\cU{\mathcal{U}}

\def\cX{\mathcal{X}}
\def\cY{\mathcal{Y}}

\def\mE{\mathbb{E}}

\def\fprod#1{\left\langle#1\right\rangle}
\def\norm#1{\left\|#1\right\|}

\newcommand{\reals}{\mathbb{R}}

\usepackage{amssymb}
\usepackage{pifont}

\usepackage{todonotes}

\title{\LARGE \bf
Variance-reduction for Variational Inequality Problems with Bregman Distance Function
}

\author{
  \begin{tabular}{c}
    Zeinab Alizadeh \textsuperscript{*},  Erfan Yazdandoost Hamedani\textsuperscript{*}, and Afrooz Jalilzadeh\footnote{Department of Systems and Industrial Engineering, The University of Arizona, Tucson, AZ, USA.\\
     \texttt{\{zalizadeh, erfany , afrooz\}@arizona.edu}}
  \end{tabular}
}
\date{}

\onecolumn
\begin{document}

\maketitle
\thispagestyle{empty}
\pagestyle{empty}


\begin{abstract}
In this paper, we address variational inequalities (VIs) with a finite-sum structure. We introduce a novel and unified stochastic variance-reduced algorithm, utilizing the Bregman distance function, that can be applied to both monotone and non-monotone settings. We establish optimal convergence guarantees under the monotone case. For the non-monotone setting, we explore a structured class of problems that exhibit weak Minty solutions and analyze the complexity of our method, demonstrating improvements over existing approaches. Numerical experiments are provided to showcase the superior performance of our algorithm compared to state-of-the-art methods.
\end{abstract}



\section{Introduction}\label{intro}
Let ($\mathcal X, \|\cdot\|_\mathcal{X}$) be a finite-dimensional normed vector space, with dual space ($\mathcal{X^*}, \|\cdot\|_\mathcal{X^*}$). In this paper, consider a Variational Inequality (VI) problem to find $x^*\in X$ such that
\begin{align}\label{VI def}
     \langle F(x^*),x-x^*\rangle +h(x)-h(x^*)\geq 0,\quad \forall x\in X,
\end{align}
where $X \subseteq \mathcal X $ is a nonempty, closed, and convex set; $F: \mathcal{X} \to  \mathcal X^*$ is a Lipschitz continuous operator; and $h:  \mathcal{X} \to  \mathbb{R} \cup \{+\infty\}$ is a proper convex lower-semicontinuous function. VI problems have received a lot of attention recently due to their general formulation that subsumes many well-known problems including (un)constrained optimization, saddle point (SP) problems, and Nash Equilibrium games \cite{kulkarni2012revisiting,ui2016bayesian}. These formulations arise in various fields such as machine learning \cite{simonetti2022example,cinelli2021variational}, robotics \cite{yang2024enhanced}, game theory \cite{guo2021variational,parise2019variational,singh2018variational,zhang2022variational}, etc. 
Developing efficient algorithms for solving large-scale VIs can immensely contribute to the computational advancement of these fields. 

One of the main structures commonly arise in a large-scale setting is when the operator $F$ consists of a large (finite) number of components, i.e., $F=\frac{1}{n}\sum_{i=1}^n F_i$. In this setting, when $n$ is large, evaluation of the operator $F$ (which may correspond to first-order information) can be prohibitively expensive. While stochastic methods have been extensively studied for a more general setting when the operator $F$ is represented by the expectation of a random variable, there is a significant gap in the convergence rate guarantee compared to deterministic methods. Consequently, variance reduction techniques have been developed to fill this gap to improve the complexity of deterministic methods. While variance reduction techniques have been successfully applied to minimization problems \cite{allen2016variance,kavis2022adaptive}, their naive extension to VI problems does not improve the complexity \cite{Alacaoglu2022Stochastic}. On the other hand, proximal-gradient-based methods with Bregman divergence have been shown to enhance algorithmic flexibility across various contexts, resulting in notable computational efficiencies in specific scenarios \cite{wang2023bregman,ahookhosh2021bregman}. The integration of Bregman distance with variance reduction for saddle point and VI problems poses an additional layer of complexity to algorithmic analysis. These challenges have been addressed by introducing retraction steps, a novel approach first implemented in the context of convex-concave saddle point problems \cite{yazdandoost2023stochastic} and later in \cite{Alacaoglu2022Stochastic} for monotone VIs. While prior methods rely on SVRG-based techniques, our approach introduces a novel variance-reduction algorithm inspired by SPIDER \zr{(see ; \cite{fang2018spider})}, offering a unified framework for solving both monotone and non-monotone VIs. This novel integration prompts the following research question: 

\begin{center}
\emph{Can we design a variance-reduced algorithm that uses Bregman distance to solve both monotone and non-monotone VIs within a unified framework?} 
\end{center}

In this paper, we answer the above question. Specifically, we introduce a novel variance-reduced method capable of utilizing Bregman distance and demonstrate its complexity results. Moreover, we show that under mild assumptions, our method can also handle the case where $F$ is non-monotone, providing a unified approach for both monotone and non-monotone VIs.

\subsection{Related work}
\zr{In the past few decades, there have been various methods developed for solving VIs with convergence rate guarantee, such as Extragradient (EG)/Mirror-Prox (MP)  \cite{juditsky2011solving,nemirovski2004prox}, forward-backward-forward (FBF) \cite{tseng2000modified}, dual extrapolation \cite{nesterov2007dual}, or reflected gradient/forward-reflected-backward (FoRB) \cite{malitsky2020forward}. More specifically, FBF improves the classical forward-backward scheme by incorporating an extra forward step, and relaxing the need for cocoercivity assumption. FoRB, on the other hand, simplifies the update rule by using only one forward and one backward step while still maintaining convergence under monotonicity, making it suitable for problems where full extragradient steps are computationally expensive.}\zr{While these methods differ in their use of operator calls and projections, } they typically demonstrate a complexity of $\cO(1/\epsilon)$ when dealing with monotone VIs in deterministic settings. The significance of stochastic VIs has grown, particularly in scenarios involving uncertainty, where methods such as Sample Average Approximation (SAA) and Stochastic Approximation (SA) come into play. SAA involves approximating the expected value of stochastic mapping by averaging over a large number of samples, while SA employs a (mini-batch) sample at each iteration. Notably, the stochastic variants of these methods are characterized by a complexity of $\cO(1/\epsilon^2)$ when addressing monotone VIs \cite{nemirovski2009robust,juditsky2011solving}.

In the optimization literature with a finite-sum objective function, variance reduction techniques, pioneered by classical works such as \zr{SVRG \cite{johnson2013accelerating}, SAGA \cite{defazio2014saga}, SARAH \cite{nguyen2017sarah}, and SPIDER \cite{fang2018spider}}, aim to improve deterministic methods. They achieve this by offering unbiased estimators of gradients while minimizing the variance of the error in gradient estimation. \zr{ These algorithms obtain faster convergence rates--often linear under strong convexity--by using recursive update rules or control variates to refine noisy gradient estimates based on previously computed information, while preserving the low per-iteration computational cost characteristic of stochastic gradient methods.}  Recently, there has been a growing interest in developing variance-reduced techniques for saddle point and VI problems, which we discuss next. 

\noindent\textbf{SP problems:} There have been various attempts to incorporate variance reduction techniques into primal-dual algorithms for solving SP problems, mostly when the objective function is strongly-convex strongly-concave such as \cite{palaniappan2016stochastic,lian2017finite,zhang2019composite,devraj2019stochastic}, and a few others consider a more general setting--see  \cite{luo2021near} and the reference therein. In particular, under a convex-concave setting and Euclidean normed vector space, the complexity of $\mathcal{O}(\tfrac{\sqrt{n}}{\epsilon}\log(1/\epsilon))$ has been shown in \cite{luo2021near}. Later, Yazdandoost Hamedani and Jalilzadeh \cite{yazdandoost2023stochastic} developed stochastic variance-reduced accelerated primal-dual (SVR-APD) algorithm equipped with Bregman distance and demonstrated the complexity of $\mathcal{O}(\sqrt{n}/\epsilon)$ matching the lower-bound obtained in \cite{han2024lower}.

\noindent\textbf{VI problems:} 
Several notable papers explore the variance reduction techniques to address monotone VIs with finite-sum structure \cite{pichugin2024optimal,song2023cyclic,Alacaoglu2022Stochastic}. 
In particular, authors in \cite{song2023cyclic} proposed a cyclic coordinate dual averaging with extrapolation (CODER) method for solving VIs of the form \eqref{VI def}. Assuming a coordinate-friendly structure and monotone operator, their method attains a complexity of $\cO(\sqrt{m}/\epsilon)$ where $m$ is the number of coordinate blocks. Furthermore, they proposed a variance-reduced variant of their method (VR-CODER) by combining an SVRG-type update leading to a three-loop algorithm. This method can achieve a complexity of $\mathcal{O}(\tfrac{\max\{\sqrt{n},\sqrt{m}\}}{\epsilon})$ to find an $\epsilon$-gap in monotone setting.
Considering a non-Euclidean setting, a variance-reduced extra-gradient method with the Bregman distance function is proposed in \cite{Alacaoglu2022Stochastic} for the class of monotone VIs. The proposed algorithm has a double-loop structure and using a retraction step complexity of $\cO(\sqrt{n}/\epsilon)$ to find an $\epsilon$-gap is achieved. Furthermore, using the same technique variance-reduced version of different algorithms including forward-backward-forward and forward-reflected-backward methods in Euclidean setup is studied -- see Table \ref{tab:2} for further comparisons of existing methods.

Exploring beyond the realm of monotone VIs, different classes of non-monotone VIs, including two-sided PL condition, pseudo monotonicity,  negatively comonotonicity, cohypomonotonicity, and weak minty VI (MVI), are studied in the literature \cite{bauschke2021generalized,iusem2017extragradient,yang2020global,pethick2023escaping,pethick2023solving,pethick2024stable}. Among these settings, the weak MVI assumption appears to be the most general existing assumption. 
With few exceptions \cite{cai2022accelerated,pethick2023escaping,alacaoglu2024extending,pethick2023solving,Alacaoglu2022Stochastic}, much of the existing research primarily investigates specific instances of \eqref{VI def}, such as when $h=0$ and $X=\mathbb{R}^n$--as evidenced in \cite{diakonikolas2021efficient,bohm2022solving,choudhury2024single}. Among these exceptions, \cite{cai2022accelerated,pethick2023escaping,alacaoglu2024extending} have attained a complexity of $\mathcal{O}(1/\epsilon^2)$ in the deterministic scenario, whereas solutions for \eqref{VI def} in a stochastic setting remain relatively scarce. Notably, \cite{pethick2023solving} introduces a bias-corrected stochastic extragradient (BC-SEG+) algorithm, guaranteeing convergence with a complexity of $\mathcal{O}(1/\epsilon^4)$. In a concurrent work to ours, Alacaoglu et al. in \cite{alacaoglu2024extending} introduce multi-loop inexact variants of Halpern and Krasnoselski-Mann (KM) iterations, achieving complexities of $\mathcal{O}(\log(1/\epsilon)/\epsilon^2)$ and $\mathcal{O}(\log(1/\epsilon)/\epsilon^4)$ for deterministic and stochastic settings, respectively -- see Table \ref{tab:nonmontone} for comparison of existing methods. 

To the best of our knowledge, no existing method provides a unified approach that can effectively handle Bregman distance functions for both monotone and non-monotone variational inequalities.

\begin{table}[!h]
\centering
\caption{Compression of algorithms for monotone VI problems}
{\begin{tabular}{|c|c|c|c|}
\hline
Setting                 & Ref & Bregman  & Complexity \\ \hline
Deterministic                    &     \begin{tabular}[c]{@{}c@{}} FBF \cite{tseng2000modified}\\FoRB \cite{malitsky2020forward} \\Prox-method \cite{nemirovski2004prox} \\ APD \cite{hamedani2021primal}\end{tabular} &     \begin{tabular}[c]{@{}c@{}} \ding{55} \\ \ding{55} \\ \checkmark \\ \checkmark  \end{tabular}       &     $\mathcal{O}(n/\epsilon)$  \\ \hline
Stochastic                      &  \begin{tabular}[c]{@{}l@{}} SA,SAA \cite{nemirovski2009robust}\\ Mirror-Prox \cite{juditsky2011solving} \end{tabular}   &    \checkmark         &     $\mathcal{O}(1/\epsilon^2)$   \\ \hline
\multirow{2}{*}{Finite-sum} &   \begin{tabular}[c]{@{}c@{}} SVRG \cite{yazdandoost2023stochastic} \\ Extra gradient \cite{Alacaoglu2022Stochastic} \\ VR-MP \cite{Alacaoglu2022Stochastic}\\ VR-CODER  \cite{song2023cyclic} \end{tabular}  &    \begin{tabular}[c]{@{}l@{}} \checkmark \\ \ding{55} \\ \checkmark \\ \ding{55} \end{tabular}              &    \multirow{2}{*}{ $\mathcal{O}(n+\frac{\sqrt{n}}{{\epsilon}})$}  \\ \cline{2-4} 
     &   This paper  &   \checkmark        &  $\mathcal{O}(n+\frac{\sqrt{n}}{{\epsilon}})$    \\ \hline
\end{tabular}}
\label{tab:2}
\end{table}

\begin{table}[!h]
\centering
\caption{Compression of algorithms for non-monotone VI problems under weak MVI assumption }
{\begin{tabular}{|c|c|c|c|}
\hline
Setting                 & Ref               & Bregman                  & Complexity      \\ \hline
Deterministic                    &  \begin{tabular}[c]{@{}c@{}}   OG\cite{cai2022accelerated}
\\AdaptiveEG+ \cite{pethick2023escaping} \\inexact-KM\cite{alacaoglu2024extending}  
\end{tabular}                     &     \begin{tabular}[c]{@{}c@{}}   \ding{55}\\ \ding{55} \\ \ding{55} 
\end{tabular}                                   &       \begin{tabular}[c]{@{}c@{}} $\mathcal{O}(n/\epsilon^2)$ \\ $\mathcal{O}(n/\epsilon^2)$ \\ $\mathcal{O}(n\log(1/\epsilon)/\epsilon^2)$ 
\end{tabular}                   \\ \hline
Stochastic                      &   \begin{tabular}[c]{@{}c@{}}  BC-SEG+ \cite{pethick2023solving}\\
 RAPP\cite{pethick2024stable} \\ Inexact-KM\cite{alacaoglu2024extending}   \end{tabular}                &    \begin{tabular}[c]{@{}c@{}}   \ding{55}\\ \ding{55} \\ \ding{55} 
\end{tabular}                  &    \begin{tabular}[c]{@{}c@{}}  $\mathcal{O}(1/\epsilon^4)$  \\ $\mathcal{O}(1/\epsilon^4)$  \\  $\mathcal{O}(\log(1/\epsilon)/\epsilon^4)$  
\end{tabular}                \\ \hline
Finite-sum & This Paper    &          \checkmark (Lip.)                            &                         $\mathcal{O}(n+\frac{1}{\epsilon^2})$      \\ \hline
\end{tabular}}
\label{tab:nonmontone}
\end{table}

\subsection{Contribution}
Motivated by the absence of a unified method capable of adapting a Bregman distance function to solve both monotone and non-monotone VIs with a finite-sum structure, we develop a \eyh{unified stochastic Variance-Reduced Forward Reflected Moving Average Backward method (VR-FoRMAB) algorithm. To contend with the challenge of using a Bregman distance function with a variance reduction, we proposed a novel stochastic operator estimation as well as a new momentum term based on a nested averaging technique. In the monotone setting, we demonstrate that our proposed method can achieve the optimal complexity of $\cO(n+\frac{\sqrt{n}}{\epsilon})$. 
Furthermore, we analyze our proposed algorithm in a non-monotone setting under a weak MVI assumption. 
While existing deterministic algorithms exhibit a complexity that scales at $n$ with the number of operator components (i.e., achieving a complexity of $\mathcal{O}(n/{\epsilon^2})$) their stochastic counterparts typically entail a complexity of $\mathcal{O}(1/\epsilon^4)$. 
Bridging this gap, we show that under weak MVI assumption and Lipschitz continuity of the Bregman distance generating function, our proposed method can achieve a complexity of $\cO(n+\frac{1}{\epsilon^2})$ for the first time, marking a significant improvement over existing methods. More specifically, considering the effect of the Lipschitz constant on the bound, our complexity result is $\cO(n+\frac{L^2}{\epsilon^2})$ where $L$ denotes the mean-square Lipschitz constant (see Assumption \ref{lip def}) while the complexity of deterministic methods is $\cO(\frac{nL_F^2}{\epsilon^2})$\footnote{$L_F$ is the Lipschitz constant of the operator $F$, and it can be as large as $\sum_{i=1}^n L_i$ where $L_i$ denotes the Lipschitz constant of an individual operator $F_i$ for each $i$. Note that if $L_i$ values are of the same order, then $L_F$ and $L$ are both $\cO(n)$. }. }


\subsection{Application}
From the application perspective, VIs manifest across diverse domains, spanning machine learning, signal processing, image processing, and finance. In particular, problems arising in machine learning include generative adversarial networks \cite{goodfellow2014generative}, reinforcement learning \cite{daskalakis2020independent}, and distributionally robust learning \cite{namkoong2016stochastic}. 
Below, we briefly outline some intriguing examples.\\
\textbf{Distributionally robust optimization (DRO):}
Define $\ell_i(u)=\ell(u,\xi_i)$, where $\ell:\mathcal X\times \Omega\to\mathbb R$ is a loss function possibly nonconvex with be a probability
space $(\Omega, {\cal F}, \mathbb P)$ where $\Omega=\{\xi_1,\hdots,\xi_n\}$. DRO examines the worst-case performance under uncertainty to determine solutions with a particular confidence level \cite{namkoong2016stochastic}. This problem can be stated in the following way:
\begin{align}\label{DRO}
    \min_{u \in \mathcal X}\max_{y\in  \cal P}\quad \mathbb E_{\xi \sim \mathbb P}[\ell(u,\xi_i)] = \sum_{i=1}^n y_i \ell_i(u),
\end{align}
where $\cal P$, represents the uncertainty set, for instance $\cal{P}$ $ = \{ y \in \Delta_n: V(y,\tfrac{1}{n}\mathbf 1_n) \leq \rho \}$ is an uncertainty set considered in different papers such as \cite{namkoong2016stochastic}, and $V (Q,P)$ denotes the divergence measure between two sets of probability measures $Q$ and $P$ and $\Delta_n \triangleq \{ y=[y_i]_{i=1}^n \in \mathbb{R}_+^n | \sum_{i=1}^n y_i =1  \} $ represents the simplex-set. By assuming $V(y,\tfrac{1}{n}\mathbf 1_n) = \sum_{i=1}^n V_i(y_i,\tfrac{1}{n}\mathbf 1_n) $. The following formulation can be obtained by relaxing the constraint on divergence in equation \eqref{DRO}
\begin{align}\label{DRO relaxation}
     \min_{\substack{u \in \mathcal{X} \\ \lambda \geq 0}}\max_{\substack{y = [y_i]_{i=1}^n \\ y \in \Delta_n } }  \sum_{i=1}^n\big(y_i \ell_i(u) - \tfrac{\lambda}{n}(V_i(y_i,\tfrac{1}{n}\mathbf 1_n)-\tfrac{\rho}{n})\big).
\end{align}
Let $x = [u^T \lambda]^T$, $L_i(x,y) = ny_i\ell_i(x)-\lambda(\frac{1}{2}V_i(y_i,\tfrac{1}{n}\mathbf 1_n)-\tfrac{\rho}{n})$, $h_1(x) = \mathbb I_{X \times \mathbb R_{+}}(u,\lambda)$ and $h_2(y) = \mathbb I_{\Delta_n}(y) $, \zr{where $\mathbb{I}_X(\cdot)$ denotes the indicator function of set $X$}, then the previous equation will be a special case of SP problem which can be formulated as a VI by setting $F(z) = F(x,y) =  \begin{bmatrix} \nabla_x L(x,y)\\ -\nabla_y L(x,y) \end{bmatrix}$ and $h(z) = h_1(x)+h_2(y) $. 
\\
\textbf{Von Neumann’s Ratio Game:} 
Consider a two-player game where the payoff function is of the form ${\langle x,Ry \rangle}/{\langle x,Sy \rangle}$, where $x,y$ are mixed-strategy vectors and $R\in\reals^{n\times m}$, $S\in\reals_+^{n\times n}$ are matrices. Such a payoff function arises in stochastic games, economic models of an expanding economy, and some nonzero-sum game formulations \cite{neumann1945model}. The goal is then to solve the following minimax problem
\begin{align}\label{matrix game def}
     \min_{x \in X} \max_{y\in  \Delta_m} \frac{\langle x,Ry \rangle}{\langle x,Sy \rangle},
\end{align}
where $X\subseteq \reals^n$ is a convex and compact set, and we assume that 
$\langle x,Sy \rangle > 0 $, for all $(x,y)\in X\times \Delta_m$. Note that problem \eqref{matrix game def} can be formulated as \eqref{VI def} which has a solution \cite{neumann1945model}.
Moreover, this problem has been shown to satisfy the weak MVI condition \cite{daskalakis2020independent}.

\subsection{Organization of the paper}
In the next section, we begin by establishing the essential notations, lemmas, and key definitions. Then, in Section \ref{Sec:Prop} we introduce our proposed algorithm VR-FoRMAB and provide the convergence analysis proving the main results in Section \ref{sec:convergence analysis}. In particular, we discuss how VR-FoRMAB can be used to solve monotone and non-monotone VI problems. Later, in section \ref{sec:experiments} we implement our proposed method to solve instances of DRO and a two-player zero-sum game and compare them with state-of-the-art methods.

\section{Preliminaries}
In this section, we provide some fundamental definitions and state the lemmas that we need for the convergence analysis. Throughout the paper, $\|\cdot\|_2$ denotes the Euclidean norm. 
\begin{definition}\label{bregman def}
    Let $\psi_{\mathcal X}: \mathcal{X} \xrightarrow{} \mathbb{R}$ be a continuously differentiable function on $\textbf{int}(\textbf{dom } h)$. In addition, $\psi_\mathcal{X}$ is strongly convex with respect to $\|.\|_\mathcal{X}$. The Bregman distance function corresponding to the distance-generating function $\psi_{\mathcal X}$ is defined as $\mathbf{D}_X (x,\bar x) \triangleq \psi_{\mathcal X}(x)-\psi_{\mathcal X}(\bar x)-\langle \nabla\psi_{\mathcal X}(\bar x),x-\bar x\rangle$ for all $x \in X$ and $\bar x \in X^\circ \triangleq X \cap \textbf{int}(\textbf{dom } h)$. Moreover we define the Bregman diamateres $B_X\triangleq \sup_{x\in \tilde X_R}\bD_X(x,x_0)$.   
\end{definition}

\begin{definition}\label{def D_k}
    For a given sequence $\{x_j\}_{j\geq 0}\subset \cX$, an index $k\in\{1,\hdots,n\}$, and $q\in\mathbb{Z}_+$, let $D_X^k(x) \triangleq \frac{1}{q}\sum_{j=(i_k-1)q+1}^{i_kq} \mathbf{D}_X(x,x_j)$ for any $x\in\cX$ where $i_k\triangleq \lfloor k/q \rfloor$. 
\end{definition}

\begin{lemma}\label{lemma 5 similar}
    Let $(\mathcal{U}, \|\cdot\|_\mathcal{U})$ be a finite-dimensional normed vector space with the dual space $(\mathcal{U}^*, \|\cdot\|_{\mathcal{U}^*})$, $f: \mathcal{U} \to \mathbb{R} \cup \{+\infty\}$ be a closed convex function, ${U} \subset \mathcal{U}$ is a closed convex set, and $\phi: \mathcal{U} \to \mathbb{R}$ be a distance-generating function which is continuously differentiable on an open set containing $\textbf{dom }f$ and is $1$-strongly convex with respect to $\|\cdot\|_\mathcal{U}$, and $\mathbf D_U: U \times (U \cap \textbf{dom }f) \to \mathbb{R}$ be a Bregman distance function associated with $\phi$. Then, the following result holds:\\    
    \text{(a)} For all $x \in U$ and $y, z \in U \cap \textbf{dom }f$, $\langle \nabla \phi(z)-\nabla \phi(\za{y}) , x-z \rangle = \mathbf D_U(x, y) -\mathbf D_U(x, z) - \mathbf D_U(z, y)$. \\
    \text{(b)} Given $\bar x \in U \cap \textbf{dom }f$, $s \in \mathcal{U}$, and $t > 0$, let $x^+ = \arg\min_{x \in U} f(x) - \langle s, x \rangle + t \mathbf D_U(x, \bar{x})$. Then, for all $x \in U$, the following inequality holds:
    \begin{align}\label{lemma 5 part a}
        f(x) \geq f(x^+)+ \langle s, x-x^+ \rangle +t \langle \nabla \phi(\bar x)-\nabla \phi(x^+), x - x^+\rangle.
    \end{align}
    \text{(c)} Given the update of $x^+$ in (a), for all $x \in  U$ the following inequality holds:
    \begin{align}\label{lemma 5 part c}
        f(x^+)-f(x)+ \langle s, x-\bar x \rangle \leq  t \left(\mathbf D_U(x, \bar x)- \mathbf D_U(x, x^+)\right) +\frac{1}{2t}\|s\|_\mathcal{U^*}^2. 
    \end{align}
\text{(d)} Assuming $\phi(\cdot)$ is a closed function, then $\nabla \phi^*(\nabla \phi(x)) = x$, for all $x \in \textbf{dom } \nabla \phi \subset \mathcal{U}$; and $\nabla \phi(\nabla \phi^*(y)) = y$, for all $y \in \textbf{dom } \nabla \phi^* \subset \mathcal U^*$.

\end{lemma}
\blue{\begin{proof}
Part $(a)$ follows from the definition of the Bregman distance function. To prove part $(b)$, one can apply \cite[Property 1]{tseng2008accelerated} combined with part $(a)$. Part $(c)$ follows by adding $\fprod{s,\bar{x}-x}$ to both sides of \eqref{lemma 5 part a} and using Young's inequality as follows:
\begin{align*}
f(x)+\fprod{s,\bar{x}-x} &\geq f(x^+)+ \langle s, \bar{x}-x^+ \rangle +t \langle \nabla \phi(\bar x)-\nabla \phi(x^+), x - x^+\rangle\\
&\geq f(x^+)-\frac{t}{2}\|\bar{x}-x^+\|^2_\cU-\frac{1}{2t}\|s\|^2_{\cU^*}+t \langle \nabla \phi(\bar x)-\nabla \phi(x^+), x - x^+\rangle\\
&\geq f(x^+)-\frac{1}{2t}\|s\|^2_{\cU^*}-t\left(\mathbf D_U(x, \bar x)- \mathbf D_U(x, x^+)\right),
\end{align*}
where in the last inequality, we used the 1-strong convexity of the Bregman distance and part $(a)$. 
Finally, part $(d)$ follows from \cite[Theorem 26.5]{rockafellar1997convex}.
\end{proof}}

\eyh{
\begin{lemma}\label{lem:aux}
Let $\{x_k\}_{k\geq 0}\subseteq \za{U}$ be a sequence and define $\tilde x_k\triangleq \frac{1}{q}\sum_{j=(i_k-1)q+1}^{i_kq}x_j$ for any $k\geq 0$. Then, for any $k\geq 0$ we have that
\begin{subequations}
\begin{align}
\label{bound for u tilde}
& \ey{\tfrac{1}{2}}\norm{x_{k+1}-\tilde x_k}_\cU^2\leq D^k_\za{U}(x_{k+1}),\\
\label{compare D^k and D}
&\sum_{k=0}^{K-1}D^k_\za{U}(x)\leq \sum_{k=0}^{K-1}\bD_\za{U}(x,x_k),\quad \forall x\in \za{U}.
\end{align}
\end{subequations}
\end{lemma}}
\begin{proof}

Using the definition of $\tilde x_k$ and the fact that for any nonnegative sequence $\{a_i\}_{i=1}^m$, $(\sum_{i=1}^m a_i)^2\leq m\sum_{i=1}^m a_i^2$, the following can be obtained:
\begin{align*}
      \norm{x_{k+1}-\tilde x_k}_\cU^2 &= \Big\|{x_{k+1}- \frac{1}{q}\sum_{j=(i_k-1)q+1}^{i_kq}x_j}\Big\|_\cU^2 \\
     & = \Big\|{\frac{1}{q}\sum_{j=(i_k-1)q+1}^{i_kq}(x_{k+1}- x_j)}\Big\|_\cU^2\\
     &\leq \frac{1}{q}\sum_{j=(i_k-1)q+1}^{i_kq}\norm{x_{k+1}- x_j}_\cU^2.
\end{align*}
The result in \eqref{bound for u tilde} follows from the definition of $D_U^k(x)$. 

For the second inequality, by using the definition of $D^k_U(x)$, for any $x\in \za{U}$ then we have,
\begin{align*}
    \sum_{k=0}^{K-1}D^k_\za{U}(x) &=  \sum_{k=0}^{K-1} \frac{1}{q}\sum_{j=\eyh{[(i_k-1)q+1]_+}}^{i_kq}D_\za{U}(x,x_j) \\
    &= \eyh{\bD_\za{U}(x,x_0)+\sum_{j=1}^{q}\bD_\za{U}(x,x_j)+\dots+\sum_{j=(i_{K-1}-1)q+1}^{i_{K-1}q}\bD_\za{U}(x,x_j)}\\
    & \eyh{= \sum_{j=0}^{i_{K-1}q} \bD_\za{U}(x,x_j)}\eyh{\leq \sum_{k=0}^{K-1} \bD_\za{U}(x,x_k).}
\end{align*}
\end{proof}

\section{\zr{Algorithmic Framework}}\label{Sec:Prop}
In this section, we propose our novel unified variance-reduced algorithm adaptable to both monotone and non-monotone scenarios. After introducing the algorithm, we divide the convergence analysis into two main sections by laying out the underlying assumptions necessary to achieve a convergence rate guarantee.  
Our goal is to develop an efficient algorithm for solving \eqref{VI def} in which we assume that sample operators are Lipschitz continuous on average. In particular, we consider the following standard assumption.
\begin{assumption}\label{lip def}
Let $F_S(x)\triangleq \frac{1}{S}\sum_{i\in\cS}F_i(x)$ for some $\cS\subset\{1,\hdots,n\}$. 
    There exists $L>0$ such that $\mathbb E[\norm{F_S(x)-F_S(\bar x)}_{\cX^*}^2]\leq L^2 \norm{x-\bar x}_\cX^2$ for any $x,\bar x\in\cX$.
\end{assumption}
Moreover, $h$ is possibly nonsmooth whose Bregman proximal operator, i.e., $\argmin_{x \in X}\break\{ h(x)+ \langle v, {x} \rangle + \tfrac{1}{\sigma} \mathbf{D}_X(x,\bar x)\}$ for some $\bar x,v\in\cX$ and $\sigma>0$, can be computed efficiently. \blue{The primary motivation for introducing Bregman distances in our algorithm design is the ability to better capture the geometry of the feasible set, especially in structured, non-Euclidean settings. For example, the effectiveness of using the negative entropy Bregman distance for projection onto the simplex set has been studied in \cite{krichene2015efficient,mao2023bregman,shi2017bregman}.} 

Variance reduction techniques stand out as a cornerstone in revolutionizing stochastic methods and were first introduced to address large-scale finite-sum minimization problems. Several methods including SAG, SVRG, and SPIDER \cite{schmidt2017minimizing,johnson2013accelerating,fang2018spider} among others were introduced to reduce the per-iteration cost. \eyh{Unfortunately, the extension of these methods to a more general setting of using the Bregman distance function or VIs is not trivial. Among these techniques, SPIDER method generates a progressively improving sequence of gradient estimates by sampling the gradients and periodically computing the full gradient when $\text{mod}(k,q)=0$ for some integer $q>0$. More precisely, if we let $v_{k-1}$ be the estimator of the gradient map $\nabla^{k-1}$ at iteration $k-1$, then the next estimator can be constructed using sample gradients $\nabla^k_S$ and $\nabla^{k-1}_S$ as follows: $v_k=v_{k-1}+\nabla_S^k-\nabla_S^{k-1}$. Note that if $\mE[v_{k-1}\mid x_{k-1}]=\delta^{k-1}$, then $\mE[v_k\mid x_k]=\delta^k$. Inspired by this approach, we introduce a new recursive update. In particular, to accommodate the Bregman distance function we introduce a new retraction step with parameter $\gamma_k\in [0,1]$ to find $\hat x_k$ such that $\grad\psi_\cX(\hat x_k)=(1-\gamma_k)\grad\psi_\cX(x_k)+\gamma_k s_k$ where $s_k$ is a moving average of iterates in the mirror space (see line 6 of Algorithm \ref{alg: VI-spider+}). Moreover, the next iterate point is found through a Bregman proximal step by moving along a direction that is a convex combination of $F$ at $x_k$ and the moving average of iterates $\tilde x_k$, i.e., $(1-\beta)F(x_k)+\beta F(\tilde x_k)$ plus a novel momentum based on the current and past operators denoted by $r_k\triangleq F_{S}(x_k)-(1-\beta)F_{S}({x}_{k-1})-\beta F_{S}(\tilde x_{k-1})$. Consequently, we update the estimate of this direction denoted by $v_k$ as follows:
\begin{equation}\label{eq:vk-update}
    v_k=v_{k-1}+(1-\beta)(F_S(x_k)-F_S(x_{k-1})),
\end{equation}
for some $\beta\in [0,1]$. Indeed, we can show that if $v_{k-1}$ is an unbiased estimator of $(1-\beta)F(x_{k-1})+\beta F(\tilde x_{k-1})$, then $v_k$ is an unbiased estimator of $(1-\beta)F(x_k)+\beta F(\tilde x_k)$ and the variance of error of this estimation depends on the cumulative distance of consecutive iterates -- see Lemma \ref{lemma: spider+}. 
Furthermore, the direction $v_k$ and momentum term $r_k$ are updated using the full batch every $q$ iteration. The steps of the proposed method are shown in Algorithm \ref{alg: VI-spider+}.}

\begin{algorithm}
\caption{VR-FoRMAB}\label{alg: VI-spider+}
\begin{algorithmic}[1] 
 \STATE \textbf{Input}: $\eyh{x}_0 \in {\mathcal X}$ , $S\in\{1,\hdots,n\}$, $q\in\{1,\hdots,n\}$, $\beta \in [0,1]$, \zr{$\theta_k = 1, \gamma_k \in (0,1)$ and $\sigma_k>0$}
 \FOR{$k = 0,\dots, K-1$}
\STATE $i_k \gets \lfloor k/q \rfloor$\;
\IF{$\text{mod}(k,q) = 0$}
\STATE $\tilde x_k \gets \frac{1}{q}\sum_{j=k-q+1}^{k} x_j$\;
      \STATE       $s_k\gets \frac{1}{q}\sum_{j=k-q+1}^{k}\grad\psi_{\cX}(x_j)$\;
        \STATE     $v_k\gets (1-\beta) F(x_k)+ \beta F(\tilde x_k)$\;
       \STATE      $\za{r}_k \gets \eyh{F(x_k)-(1-\beta)F({x}_{k-1})-\beta F(\tilde x_{k-1})} $\;
\ELSE
 \STATE $\tilde x_k\gets \tilde x_{i_kq}$\;
	 \STATE	  Draw $S$ samples $\cS\subset\{1,\hdots,n\}$ and let $F_S\triangleq \frac{1}{S}\sum_{i\in\cS} F_i$ \;
        \STATE     $v_k \gets v_{k-1}
            +(1-\beta) (F_{S}(x_k)-F_{S}(x_{k-1}))$\; 
        \STATE       $\za{r}_k \gets \eyh{F_{S}(x_k)-(1-\beta)F_{S}({x}_{k-1})-\beta F_{S}(\tilde x_{k-1})} $\;
\ENDIF
  \STATE $\hat x_k\gets \grad \psi_\cX^*((1-\gamma_k)\grad\psi_\cX(x_k) + \gamma_k s_k)$\;
     \STATE     $x_{k+1} \leftarrow \argmin_{x \in X}\{ h(x)+ \langle v_k+\theta_k \za{r}_k, \za{x} \rangle + \tfrac{1}{\sigma_k} \mathbf{D}_X(x,\eyh{\hat x}_k)\} $\;
 \ENDFOR
\end{algorithmic}
\end{algorithm}
Next, we introduce key definitions to streamline notation and establish essential properties of the variance reduction technique outlined in \eqref{eq:vk-update} in the consequent lemmas. These results serve as foundational elements crucial for demonstrating the convergence rate of the proposed algorithm in both monotone and non-monotone settings.

\begin{definition} \label{def param}
    For any $k\geq 0$, let $\delta_k \triangleq v_k -((1-\beta)F(x_k)+\beta F(\tilde x_k))$, $\bar r_k \triangleq F(x_k)-(1-\beta)F({x}_{k-1})-\beta F(\tilde x_{k-1})$, and 
    \blue{\begin{align*}
    r_k\triangleq \begin{cases}F_{S}(x_k)-(1-\beta)F_{S}({x}_{k-1})-\beta F_{S}(\tilde x_{k-1})& if \quad \text{mod}(k, q) \neq 0\\
    \bar r_k& o.w.
    \end{cases}
    \end{align*}} .
\end{definition}

\begin{lemma}\label{bound base on norm 2 delta}
Let $\{x_k\}_{k\geq 0}$ be the sequence generated by Algorithm \ref{alg: VI-spider+}.
We define some auxiliary sequence $\{w_k^x,w_k^r\}\subset \cX$, for $k\geq 0$ as follow: 
\begin{subequations}\label{eq:diff u w}
\begin{align}
   &w_{k+1}^x\leftarrow \argmin_{x \in X} \{ -\langle \za{\delta_k},x \rangle + \tfrac{1}{\eta_{k}} \mathbf{D}_X(x,w_{k}^x)\}\label{dif u} \\ 
    &w_{k+1}^\za{r} \leftarrow \argmin_{x \in X} \{\langle r_k-\bar r_{k},x \rangle + \tfrac{1}{\eta_{k}} \mathbf{D}_X(x,w_{k}^\za{r})\} \label{dif w}.
\end{align}
\end{subequations} \\
Then, for any $k\geq 0$, $\eta_k^1,\eta_k^2>0$, and $x\in X$,
\begin{subequations}
    \begin{align}
        \label{bound a}
        \langle \za{\delta_k} , x-x_{k+1} \rangle &\leq \tfrac{1}{\eta_{k}^1} (\mathbf{D}_X(x,w_{k}^x)-\mathbf{D}_X(x,w_{k+1}^x)) + \ey{\fprod{\delta_k,w_k^x-w_{k+1}^x+x_k-x_{k+1}}} 
        \nonumber\\
        &\quad -\tfrac{1}{\eta_k^1}\bD_X(w_{k+1}^x,\zi{w}_k^x)+\langle \za{\delta_k} , w_{k}^x-x_{k} \rangle,\\
        \label{bound c}
        \eyy{\langle {r}_k-\bar r_{k} , x-x_{k+1} \rangle} &\leq \tfrac{1}{\eta_{k}^2} (\mathbf{D}_X(x,w_k^\za{r})-\mathbf{D}_X(x,w_{k+1}^\za{r}))+\eyy{\eta_{k}^2\|{r}_k-\bar r_{k}\|_{\mathcal{X}^*}^2} + \langle r_k-\bar r_{k} ,w_{k}^r -x_k\rangle\nonumber \\ 
        &\quad  + \eyy{\tfrac{1}{\eta_k^2}\bD_X(x_{k+1},x_k)}. 
    \end{align}
\end{subequations}
\end{lemma}
\begin{proof}
We begin the proof of inequality in \eqref{bound a} by splitting the inner product into $\langle \za{\delta_k} , x-x_{k} \rangle $ and $\langle \za{\delta_k}, x_k-x_{k+1} \rangle$ and providing an upper bound for
\ey{the first inner product}. In particular, using Lemma \ref{lemma 5 similar} parts $(a)$ and $(b)$ for the sequence in \eqref{dif u} with $s =\za{\delta_k}, t = 1/\eta_{k}^1$, and $f \equiv 0$, for any $x \in X$ we conclude that 
\begin{align} \label{eq:lemma 1 on delta}
    \nonumber\langle \za{\delta_k} , x-x_{k} \rangle &= \langle \za{\delta_k}, x-w_{k}^x \rangle + \langle \za{\delta_k}, w_{k}^x-x_{k} \rangle\\
    &\leq \tfrac{1}{\eta_{k}^1} (\mathbf{D}_X(x,w_{k}^x)-\mathbf{D}_X(x,w_{k+1}^x)-\mathbf{D}_X(w_{k+1}^x,w_k^x))+\fprod{\delta_k,\zi{w_{k+1}^x -w_k^x}}\nonumber\\
    &\quad + \langle \za{\delta_k} , w_{k}^x-x_{k} \rangle.
\end{align}
\ey{Therefore the result in \eqref{bound a} follows from \eqref{eq:lemma 1 on delta} by adding $\langle \za{\delta_k} , x_k-x_{k+1} \rangle$ to both sides.} 

The inequality in \eqref{bound c} can be shown by splitting the inner product into $\fprod{r_k-\bar{r}_k,x-x_k}$ and $\fprod{r_k-\bar{r}_k,x_{k+1}-x_k}$ and directly applying the result of Lemma \ref{lemma 5 similar} part $(c)$ on the former term with $s =r_k-\bar{r}_k, t = 1/\eta_{k}^2$, and $f \equiv 0$ while applying Young's inequality in the latter one. 
\end{proof}


\begin{lemma}\label{lemma: spider+} 
Let $\{x_k\}_{k\geq 0}$ be the sequence generated by Algorithm \ref{alg: VI-spider+} and $\{w_k^x,w_k^r\}_{k\geq 0}$ defined as in \eqref{eq:diff u w}. Then, for any $k \geq 0$, $\mE[\delta_k]=0$ and $\mE[\za{r}_k-\bar r_k]=0$. Moreover, for $k\geq 0$
\begin{subequations}
\begin{align}
\label{eq:delta-norm2 for y spider+}
\ey{\mE\left[\fprod{\delta_k,w_k^x-w_{k+1}^x+x_k-x_{k+1}}\right]}
&\leq \tfrac{(1-\beta)q}{\eta_k^4} \mE[\bD_X(w_{k+1}^x,w_k^x)]+\tfrac{(1-\beta)q}{\eta_k^3}\mE[\bD_X(x_{k+1},x_k)]
 \\
&\quad +4(1-\beta)L^2(\eta_k^4+\eta_k^3)\sum_{t=i_kq+1}^k\mE[\bD_X(x_t,x_{t-1})],\nonumber\\
 \label{eq:delta-norm2 for q spider+}
\mE\big[\norm{r_k - \bar r_k}_{\cX^*}^2\big]&\leq \ey{16}(1-\beta)^2L^2 \mE[\bD_X(x_k,x_{k-1})]+\ey{16}\beta^2 L^2 \mE[D^{k-1}_X(x_k)].
\end{align}
\end{subequations}
 
\end{lemma}

\begin{proof}
Using the definitions of $\delta_k$ and $v_k$ in Algorithm \ref{alg: VI-spider+}, for any $k\geq 0$ and $t\in\{i_kq+1,\hdots,(i_k+1)q-1\}$ we have that
\begin{align*}
    v_t =   v_{t-1}+(1-\beta)(F_S(x_t)-F_S(x_{t-1})).
\end{align*}
Adding and subtracting $(1-\beta)F(x_t)+\beta F(\tilde x_t)$ and defining $e_t\triangleq F_S(x_t)-F_S(x_{t-1})-F(x_t)+F(x_{t-1})$ leads to:
\begin{align*}
     \delta_t &=\delta_{t-1}+ (1-\beta) e_t+\beta(F(\tilde x_{t-1})-F(\tilde x_t)).
\end{align*}
Therefore, summing the above relation from $t=i_kq+1$ to $k$ and noting that $\delta_{i_kq}=0$ we obtain
\begin{align}\label{define delta}
    \delta_k&=(1-\beta)\sum_{t=i_kq+1}^k e_t+\beta(F(\tilde x_{i_kq})-F(\tilde x_k))=(1-\beta)\sum_{t=i_kq+1}^k e_t,
\end{align}
where the last equality follows from the fact that $\tilde x_t=\tilde x_{i_kq}$ for any $t\in\{i_kq+1,\hdots,(i_k+1)q-1\}$. Therefore, we immediately conclude that $\mE[\delta_k]=0$ for any $k\geq 0$. 

\ey{Next, using \eqref{define delta} followed by applying Cauchy Schwartz and triangle inequalities we obtain} 
\begin{align}\label{eq:young delta}
&\ey{\fprod{\delta_k,w_k^x-w_{k+1}^x+x_k-x_{k+1}}}\\
&\quad \leq 
\norm{\delta_k}_{\cX^*}\big(\|w_{k+1}^x-w_k^x\|_\cX+\|x_{k+1}-x_k\|_\cX\big) \nonumber\\
&\quad \leq (1-\beta)\sum_{t=i_kq+1}^k \norm{e_t}_{\cX^*}\big(\|w_{k+1}^x-w_k^x\|_\cX+\|x_{k+1}-x_k\|_\cX\big) \nonumber\\
&\quad \leq (1-\beta)\sum_{t=i_kq+1}^k\Big(\tfrac{\eta_k^4+\eta_k^3}{2}\norm{e_t}_{\cX^*}^2 + \tfrac{1}{\eta_k^4}\bD_X(w_{k+1}^x,w_k^x)+\tfrac{1}{\eta_k^3}\bD_X(x_{k+1},x_k)\Big), \nonumber
\end{align}
where the last inequality follows from the application of Young's inequality.
Moreover from the definition of $e_t$ and Assumption \ref{lip def} we have that
\begin{align}\label{bound for norm2 e^y}
\mE[\norm{e_t}_{\cX^*}^2]&=\mE\big[\big\|F_S(x_t)-F(x_t)-F_S(x_{t-1})+F(x_{t-1})\big\|^2_{\cX^*}\big] \\
&\leq 4L^2\mE[\norm{x_t-x_{t-1}}^2_\cX].\nonumber
\end{align}
Therefore, combining \eqref{bound for norm2 e^y} with \eqref{eq:young delta}, using strong convexity of the Bregman distance, and the fact that $k-i_kq\leq q$ for any $k\geq 0$, lead to the desired result in \za{\eqref{eq:delta-norm2 for y spider+}}. 

Next, recall that \blue{for any $k\geq 0$, $\bar r_k = F(x_k)-(1-\beta)F({x}_{k-1})-\beta F(\tilde x_{k-1})$, moreover, $r_k= F_{S}(x_k)-(1-\beta)F_{S}({x}_{k-1})-\beta F_{S}(\tilde x_{k-1})$ if $\text{mod}(k,q)\neq 0$ and $\bar r_k$ otherwise.} It is easy to verify that $\mE[r_k-\bar r_k]=0$ for $k\geq 0$. Define $\tilde e_k\triangleq F_S(x_k)-F_S(\tilde x_{k-1})+F(\tilde x_{k-1})-F(x_k)$, then using the triangle inequality and Assumption \ref{lip def}, one can easily demonstrate that
\begin{align}\label{eq:qk-barqk}
    \mE[\norm{r_k - \bar r_k}_{\cX^*}^2]&\leq
    2(1-\beta)^2\mE[\norm{e_k}_{\cX^*}^2] + 2\beta^2\mE[\norm{\tilde e_k}_{\cX^*}^2]\\
&\leq 8(1-\beta)^2L^2\mE[\norm{x_k-x_{k-1}}^2_\cX]+8\beta^2 L^2 \mE[\norm{x_k-\tilde x_{k-1}}^2_\cX].\nonumber
\end{align}
Then, using strong convexity of Bregman distance function and \eqref{bound for u tilde} in Lemma \ref{lem:aux} we obtain the desired result. 
\end{proof}
In the following lemma, we present a one-step analysis for VR-FoRMAB, which serves as the main component for demonstrating the rate result in section \ref{sec:convergence analysis}.
\begin{lemma}\label{bound for gap}
Let $\{x_k\}_{k\geq 0}$ be the sequence generated by VR-FoRMAB displayed in Algorithm \ref{alg: VI-spider+} initialized
from arbitrary vectors $x_0 \in {\mathcal X}$. Let $\{w_k^x,w_k^\za{r}\}$ be the auxiliary sequence defined in \eqref{dif u}-\eqref{dif w}. Suppose Assumption \ref{lip def} holds and $ \za{r}_k$, $\delta_k$, and $\bar r_k$ are defined in Definition \ref{def param}. For any $x \in X$ and $k\geq 0$ the following result holds:
\begin{align}\label{comb 3 terms spider+}
     &h(x_{k+1}) - h(x) + \langle F(x), x_{k+1}-x \rangle\\
        &\leq  A_k +\theta_k \langle\bar r_k,\za{x-x_{k}}\rangle -\langle \bar r_{k+1} , x-x_{k+1} \rangle   +  \tfrac{1-\gamma_k} {\sigma_k} (\mathbf{D}_X(x,x_k)-\mathbf{D}_X(x,x_{k+1}))\nonumber\\
        &\quad +\big(\eyy{(\tfrac{(1-\beta)q}{\eta_k^3}+\tfrac{1}{\eta_k^2})}+\za{\theta_k}(\alpha_1+\alpha_2)-\tfrac{1-\gamma_k}{\sigma_k}\big)\mathbf{D}_X(x_{k+1},x_k)\nonumber\\
        &\quad 
         +\tfrac{\gamma_k} {\sigma_k  } (D^k_X(x)-\mathbf{D}_X(x,x_{k+1})-D^k_X(x_{k+1}))\nonumber\\
         &\quad +{\theta_k}(1-\beta)^2L^2\big(\tfrac{1}{\alpha_1}+\eyy{16\eta_k^2}\big)\bD_X(x_k,x_{k-1})+{\theta_k}\beta^2L^2 \big(\tfrac{1}{\alpha_2}+\ey{16\eta_k^2}\big)D_X^{k-1}(x_k) \nonumber\\
         &\quad + \eyy{4(1-\beta)L^2(\eta_k^4+{\eta_k^3})} \sum_{t=\eyh{i_kq+1}}^k \bD_X(x_t,x_{t-1}),\nonumber
\end{align}
where $A_k\triangleq \langle\delta_k , x-x_{k+1} \rangle+\theta_k\langle r_k - \bar r_k, x-x_{k+1} \rangle -(\tfrac{(1-\beta)q}{\eta_k^3}+\tfrac{1}{\eta_k^2})\bD_X(x_{k+1},x_k) -4(1-\beta)L^2(\eta_k^4+\eta_k^3)\sum_{t=i_kq+1}^k\bD_X(x_t,x_{t-1}) -\ey{16}\theta_k\eta_{k}^2L^2\big((1-\beta)^2 \bD_X(x_k,x_{k-1})+\beta^2 D^{k-1}_X(x_k)\big)$.

\end{lemma}

\begin{proof}
Applying Lemma \ref{lemma 5 similar}  for updating the rule of $x_{k+1}$, implies that for any $ x \in X$, 
    \begin{align*}
        h(x_{k+1}) - h(x) &\leq  \langle v_k+\theta_k r_k, x-x_{k+1} \rangle + \tfrac{1}{\sigma_k} \langle  \nabla \psi_{\mathcal X}(x_{k+1}) - \nabla \psi_{\mathcal X}(\hat x_{k}), x- x_{k+1}\rangle\\
        &=  \langle v_k+\theta_k r_k, x-x_{k+1} \rangle \\
        &\quad + \tfrac{1}{\sigma_k} \langle  \nabla \psi_{\mathcal X}(x_{k+1}) - (1-\gamma_k)\grad\psi_\cX(x_k) - \gamma_k s_k, x- x_{k+1}\rangle\\
        &=  \langle v_k+\theta_k r_k, x-x_{k+1} \rangle +  \tfrac{1-\gamma_k}{\sigma_k} \langle  \nabla \psi_{\mathcal X}(x_{k+1}) - \grad\psi_\cX(x_k), x- x_{k+1}\rangle\\
        &\quad +\tfrac{\gamma_k}{\sigma_k}\langle \nabla \psi_{\mathcal X}(x_{k+1})-s_k, x- x_{k+1}\rangle\\
         &=  \langle v_k+\theta_k r_k, x-x_{k+1} \rangle +  \tfrac{1-\gamma_k}{\sigma_k} \langle  \nabla \psi_{\mathcal X}(x_{k+1}) - \grad\psi_\cX(x_k), x- x_{k+1}\rangle\\
         &\quad +\tfrac{\gamma_k}{\sigma_k q}\za{\sum_{j=(i_k-1)q+1}^{i_kq}}\langle \nabla \psi_{\mathcal X}(x_{k+1})-\nabla \psi_{\mathcal X}(x_{j}), x- x_{k+1}\rangle.
    \end{align*}
    Where in the first equality we used Lemma \ref{lemma 5 similar} part(d) and updated the rule of $\hat{x}_k$
 in Algorithm \ref{alg: VI-spider+}. Then by using the generalized three-point property of Bregman distance in Lemma \ref{lemma 5 similar} part (a) twice; we will have:
  \begin{align}
        \nonumber h(x_{k+1}) - h(x) &\leq \langle v_k+\theta_k r_k, x-x_{k+1} \rangle\\
        \nonumber& \quad+  \tfrac{1-\gamma_k} {\sigma_k} (\mathbf{D}_X(x,x_k)-\mathbf{D}_X(x,x_{k+1})-\mathbf{D}_X(x_{k+1},x_k))
        \\
        &\quad+ \tfrac{\gamma_k} {\sigma_k q} \za{\sum_{j=(i_k-1)q+1}^{i_kq}}(\mathbf{D}_X(x,x_i)-\mathbf{D}_X(x,x_{k+1})-\mathbf{D}_X(x_{k+1},x_i)),
    \end{align}
      by using the definition of $\delta_k$ \eyy{and $\bar{r}_k$}, and rearranging the terms, we have: 
    \begin{align}\label{change v to delta spider+}
         h(x_{k+1}) - h(x) &\leq \langle  \delta_k , x-x_{k+1} \rangle + \langle  (1-\beta)F(x_k)+\beta F(\tilde x_k), x-x_{k+1} \rangle  \\
        \nonumber&\quad+\eyy{\theta_k\langle  \bar{r}_k, \za{x-x_{k+1}} \rangle + \theta_k\langle  r_k-\bar{r}_k, x-x_{k+1}} \rangle\\
        &\nonumber \quad+  \tfrac{1-\gamma_k} {\sigma_k} (\mathbf{D}_X(x,x_k)-\mathbf{D}_X(x,x_{k+1})-\mathbf{D}_X(x_{k+1},x_k))
        \\
        &\quad+ \tfrac{\gamma_k} {\sigma_k  } (D^k_X(x)-\mathbf{D}_X(x,x_{k+1})-D^k_X(x_{k+1})).\nonumber
    \end{align}
     Now, by adding $\langle F(x_{k+1}), x_{k+1}-x \rangle$ to the both sides, and adding and subtracting $\langle \za{\theta_k}\bar r_k , \za{x-x_{k}} \rangle$ to the right-hand side of \eqref{change v to delta spider+} one can obtain: 
      \begin{align} \label{add and subtract grad L}
        &h(x_{k+1}) - h(x) + \langle F(x_{k+1}), x_{k+1}-x \rangle\\
        &\leq \langle  \delta_k , x-x_{k+1} \rangle + \theta_k\langle r_k - \bar r_k, \eyy{x-x_{k+1}} \rangle + \langle \bar r_{k+1} , x_{k+1} -x \rangle  +\eyy{\theta_k \langle \bar r_k, {x-x_{k}} \rangle+ \theta_k\langle  \bar{r}_k, {x_k-x_{k+1}} \rangle}\nonumber\\
         &\quad +  \tfrac{1-\gamma_k} {\sigma_k} (\mathbf{D}_X(x,x_k)-\mathbf{D}_X(x,x_{k+1})-\mathbf{D}_X(x_{k+1},x_k))
         +\tfrac{\gamma_k} {\sigma_k  } (D^k_X(x)-\mathbf{D}_X(x,x_{k+1})-D^k_X(x_{k+1})).\nonumber
    \end{align}
    Next, by using Cauchy-Schwartz inequality, followed by Assumption \ref{lip def} and Young's inequality  we obtain
    \begin{align}\label{bound by Cauchy}
        &\eyy{|\langle \bar{r}_k, {x_k-x_{k+1}} \rangle|} \\ \nonumber
        &\leq \tfrac{(1-\beta)^2L^2}{2\za{\alpha_1}}\norm{x_k-x_{k-1}}^2_\cX + \tfrac{\beta^2 L^2}{2\za{\alpha_2}}\norm{x_k-\tilde x_{k-1}}^2_\cX+\tfrac{\za{\alpha_1+\alpha_2}}{2}\norm{x_{k+1} -  x_{k}}^2_\cX\\ \nonumber
        &\leq \tfrac{(1-\beta)^2L^2}{\za{\alpha_1}}\bD_X(x_k,x_{k-1})+\tfrac{\beta^2L^2}{\alpha_2}D_X^{k-1}(x_k) + (\alpha_1+\alpha_2)\bD_X(x_{k+1},x_k),
    \end{align}
for any $\alpha_1,\alpha_2 >0$.
Note that the left-hand side of \eqref{add and subtract grad L} can be lower-bounded using monotonicty of $F$ as follows $h(x_{k+1}) - h(x) + \langle F(x_{k+1}), x_{k+1}-x \rangle\geq h(x_{k+1}) - h(x) + \langle F(x), x_{k+1}-x \rangle$. 
\eyy{Moreover, by using the previous inequality within \eqref{add and subtract grad L}, adding and subtracting the following terms (inspired by the right hand-sides of \eqref{eq:delta-norm2 for y spider+} and \eqref{eq:delta-norm2 for q spider+} containing differences of consecutive iterates)
\begin{align*}
    &(\tfrac{(1-\beta)q}{\eta_k^3}+\tfrac{1}{\eta_k^2})\bD_X(x_{k+1},x_k) +4(1-\beta)L^2({\eta_k^4}+{\eta_k^3})\sum_{t=i_kq+1}^k\bD_X(x_t,x_{t-1}) \\
    &\quad+16\theta_k\eta_{k}^2L^2\Big((1-\beta)^2 \bD_X(x_k,x_{k-1})+\beta^2 D^{k-1}_X(x_k)\Big), 
\end{align*}
in light of the definition of $A_k$ in the statement of the lemma, the desired result in \eqref{comb 3 terms spider+} can be obtained.}  
\end{proof}

\begin{remark}
Note that the retraction parameter $\gamma_k$ and damping parameter $\beta$ are introduced to control the upper bounds on the operator estimates (Lemma \ref{lemma: spider+}). This can be seen in the result of Lemma \ref{bound for gap} to create telescopic terms on the right-hand side of \eqref{comb 3 terms spider+}. 
\end{remark}

\section{Convergence Analysis}\label{sec:convergence analysis}

\subsection{Monotone VI} In this section, we demonstrate the complexity of Algorithm \ref{alg: VI-spider+} under monotonicity of $F$. We first formally state the assumptions.
\begin{assumption}\label{assum:sol-exist}
    A solution $x^*\in X$ to the VI problem in \eqref{VI def} exists.
\end{assumption}
\begin{assumption}\label{assum:monotone}
Assume that $F:\cX\to \cX^*$ is a monotone operator, i.e., for any $x,\bar x \in \cX$ $\fprod{F(x)-F(\bar x),x-\bar x}\geq 0$.
\end{assumption}
In this scenario, we consider the standard gap function defined as follows:
\begin{equation}\label{gap function}
    \cG(\bar x) \triangleq \sup_{x \in \tilde X_R}\{ \langle F(x),\bar x-x\rangle +h(\bar x)-h(x) \},
\end{equation}
where $\tilde X_R\subseteq X$ is a convex and compact set containing a solution of \eqref{VI def} to handle the possibility of unbounded domain \cite{nesterov2007dual}. 
    

In the next theorem, we provide the main result of this paper by finding a bound for the expected gap function. 

\begin{theorem}\label{Th: rate}
Let $\{x_k\}_{k\geq 0}$ be the sequence generated by VR-FoRMAB displayed in Algorithm \ref{alg: VI-spider+} initialized
from arbitrary vector $x_0 \in {\mathcal X}$. Suppose Assumptions \ref{lip def}, \ref{assum:sol-exist}, and \ref{assum:monotone} hold. For  $k\geq 0$ let $\theta_k=1$, $\gamma_k=\gamma\in (0,1)$, $\beta\in (0,1)$, and the step-size sequence \ey{$\sigma_k=\sigma\leq \min\{\tfrac{1-\gamma}{(1+9(1-\beta)+6(1-\beta)q)L},\tfrac{\gamma}{\beta L(1+4\beta/(1-\beta))}\}$}. Then, for any $K\geq 1$,  
\begin{align}
      \mE[\cG(\bar x_K)] &\leq  \frac{1}{K}\Big(\frac{1}{\sigma}+\ey{2(1-\beta)(q+2)L} \Big)B_X,
\end{align}
where $B_X\triangleq \sup_{x\in \tilde X_R}\bD_X(x,x_0)$. 
\end{theorem}
\begin{proof}
Consider the result in Lemma \ref{bound for gap} and let $\sigma_k=\sigma$, $\gamma_k = \gamma$, $\theta_k=1$, $\alpha_1=(1-\beta)L$, $\alpha_2=\beta L$, and \ey{$\Gamma\triangleq (1+16\eta^2(1-\beta)L)(1-\beta)L+4(1-\beta)qL^2({\eta^4}+{\eta^3})$} for $k\geq 0$, then summing the inequality over $k = 0, \dots K-1$, 
dividing both sides by $K$, using  \eqref{compare D^k and D}, and the fact that \eyy{$\sum_{k=0}^{K-1}\sum_{t=i_kq+1}^k\bD_X(x_t,x_{t-1})\leq q\sum_{k=0}^{K-1}\bD_X(x_k,x_{k-1})$} imply that
 \begin{align}\label{one-step-sum_spiderplus}
     &\frac{1}{K}\sum_{k=0}^{K-1} \Big( h(x_{k+1}) - h(x) + \langle F(x), x_{k+1}-x \rangle \Big)\\
        &\leq  \frac{1}{K}\za{\Big [}\sum_{k=0}^{K-1}A_k \za{-} \langle\bar r_0, x_0-x \rangle \za{+} \langle \bar r_K , x_K -x \rangle +  \tfrac{1-\gamma}{\sigma} (\mathbf{D}_X(x,x_0)-\mathbf{D}_X(x,x_{K}))\nonumber\\
        &\quad +\tfrac{\gamma}{\sigma } (\bD_X(x,x_0)-\bD_X(x,x_{K}))\za{\Big]}+\frac{1}{K}\sum_{k=0}^{K-1}\Big((1+\ey{16\eta^2\beta L})\beta L D_X^{k-1}(x_k) - \tfrac{\gamma}{\sigma }D^k_X(x_{k+1})\Big)
        \nonumber\\
         &\quad +\frac{1}{K}\sum_{k=0}^{K-1}\Big( \eyy{\Gamma} \bD_X(x_k,x_{k-1})+ (\eyy{\tfrac{(1-\beta)q}{\eta^3}+\tfrac{1}{\eta^2}}+L-\tfrac{1-\gamma}{\sigma})\bD_X(x_{k+1},x_k)\Big)\nonumber\\
         &\leq \frac{1}{K}\Big[\sum_{k=0}^{K-1}A_k \za{-} \langle\bar r_0, x_0-x \rangle + \langle \bar r_K , x_K -x \rangle +  \frac{1-\gamma}{\sigma} (\mathbf{D}_X(x,x_0)-\mathbf{D}_X(x,x_{K}))\nonumber\\
        &\quad +\frac{\gamma}{\sigma } (\bD_X(x,x_0)-\bD_X(x,x_{K}))\Big]+\frac{1}{K}\frac{\gamma}{\sigma}\Big(D_X^{-1}(x_0) - D^{K-1}_X(x_{K})\Big)
        \nonumber\\
         &\quad +\frac{1}{K}\Big(\frac{1-\gamma}{\sigma}-\eyy{L\big(1+2\sqrt{2}(1-\beta)+2(1-\beta)q\big)}\Big)\Big( \bD_X(x_0,x_{-1})-\bD_X(x_{K},x_{K-1})\Big)
         .\nonumber
\end{align}
where the last inequality follows by selecting \ey{$\eta^2 = \frac{1}{4(1-\beta)L}$}, $\eta^3=\eta^4=\frac{1}{2L}$, and using the step-size condition. 

Next, we provide an upper bound on $\langle \bar r_K , x_K -x \rangle \za{-} \langle\bar r_0, x_0-x \rangle $ by using Young's inequality and noting that $\bar r_0=0$ due to the initialization of the algorithm, i.e., $x_{-1}=x_0$. Therefore, similar to \eqref{bound by Cauchy} 
\begin{align}\label{inner pro bound q}
\langle \bar r_{K} , x_{K} -x \rangle\leq \|\bar r_K\|_{\cX^*}\|x-x_K\|_\cX  \leq {(1-\beta)L}\bD_X(x_K,x_{K-1})+\beta L D_X^{K-1}(x_K) + L\bD_X(x,x_K),
\end{align}
moreover, the left-hand side of \eqref{one-step-sum_spiderplus} can be lower-bounded using Jensen's inequality due to the convexity of $h$ and defining $\bar x_K\triangleq \frac{1}{K}\sum_{k=0}^{K-1}x_k$. Therefore, using \eqref{inner pro bound q} within \eqref{one-step-sum_spiderplus} and the fact that $\bD_X(x_0,x_{-1})=0$ and $D_X^{-1}(x_0)=0$ we conclude that
\begin{align}\label{bound with q bar k}
    &h(\bar x_K) - h(x) + \langle F(x), \bar x_K-x \rangle \\
    &\leq \frac{1}{K}\Big[\sum_{k=0}^{K-1}A_k +  \frac{1}{\sigma} \mathbf{D}_X(x,x_0)+\big(L-\frac{1}{\sigma}\big)\mathbf{D}_X(x,x_{K})\nonumber\\
        &\quad +\big(\ey{L(1+5(1-\beta)+2(1-\beta)q)}-\frac{1-\gamma}{\sigma}\big)\bD_X(x_K,x_{K-1})+\big(\beta L-\frac{\gamma}{\sigma}\big) D_X^{K-1}(x_K)\Big].\nonumber
\end{align}
Note that since $\beta\in [0,1]$ and $q\geq 1$, from the stepsize condition one clearly verify that \eyy{the last three terms in the right hand-side of the above inequality have non-positive coefficients}, hence, 
can be dropped. Moreover, from the definition of $A_K$, using Lemmas \ref{bound base on norm 2 delta} and \ref{lemma: spider+}, \eyy{and selecting $\eta^1=\frac{1}{2(1-\beta)qL}$} we can show that
\begin{align}\label{eq:exp-sum-Ak}
& \sum_{k=0}^{K-1} \mE[A_k] \\ &=\sum_{k=0}^{K-1} \mE\Big[\langle\delta_k , x-x_{k+1} \rangle+\langle \za{r}_k - \bar r_k, \eyy{x-x_{k+1}} \rangle
- \eyy{4(1-\beta)L} \sum_{t=\eyh{i_kq+1}}^k \bD_X(x_t,x_{t-1}) \nonumber \\
&\quad-\eyy{(1-\beta)(2+q)L}\bD_X(x_{k+1},x_k) - \ey{4(1-\beta)L} \bD_X(x_k,x_{k-1})-\ey{\frac{4\beta^2 L}{(1-\beta)}} D^{k-1}_X(x_k)\Big] \nonumber\\
&\leq \sum_{k=0}^{K-1}\mE\Big[\langle {\delta_k} ,w_{k}^x-x_{k} \rangle + \langle \za{r}_k-\bar r_k ,w_{k}^r-x_k \rangle \Big]+\eyy{2(1-\beta)qL}\sum_{k=0}^{K-1}\mE\Big[\bD_X(x,w_{k}^x)-\bD_X(x,w_{k+1}^x)\Big] \nonumber\\
&\quad + \ey{4(1-\beta)L}\sum_{k=0}^{K-1}\mE\Big[\bD_X(x,w_{k}^\za{r})-\bD_X(x,w_{k+1}^\za{r})\Big]\nonumber\\
& \leq \eyy{2(1-\beta)qL}\bD_X(x,w_{0}^x)+ \ey{4(1-\beta)L}\bD_X(x,w_{0}^r),\nonumber
\end{align} 
where the last inequality follows from unbiasedness of $\delta_k$ and $\za{r}_k-\bar r_k$ and dropping the non-positive terms. 
Therefore, combining \eqref{bound with q bar k} with \eqref{eq:exp-sum-Ak}  and taking the supremum over $x\in \tilde X_R$ followed by taking the expectation of both sides of the aforementioned inequality implies that
\begin{align}\label{eq:one-step-sum-2}
    \mE[\cG(\bar x_K)] \leq 
    \frac{1}{K}\Big(\frac{1}{\sigma}+\ey{2(1-\beta)(q+2)L} \Big)B_X,
\end{align}
where $B_X= \sup_{x\in \tilde X_R}\bD_X(x,x_0)$ and note that $w_0^x=w_0^r=x_0$.
\end{proof}    

\begin{corollary}\label{cor:rate-monotone}
Under the premises of Theorem \ref{Th: rate} , if the step-sizes and parameters of Algorithm \ref{alg: VI-spider+} are selected as $q=n$, $S=1$, \ey{$\beta = 1-\frac{1}{\sqrt{n}}$, $\theta_k = 1$, $\gamma\in (0,1)$, $\sigma_k= \min\{\tfrac{1-\gamma}{(1+6\sqrt{n}+9/\sqrt{n})L},\frac{\gamma}{(1+4\sqrt{n})L}\}$}, then $ \mE[\cG(\bar x_K)]  \leq  \mathcal O(\tfrac{L\sqrt{n}}{K})$. Moreover, the oracle
complexity to achieve an $\epsilon$-gap is $\cO(n+\sqrt{n}L/\epsilon)$. 
\end{corollary}
\begin{proof}
Let $q = n$, then it is easy
to verify that for any $\gamma\in (0,1)$, ${\sigma}=\cO(\frac{1}{\sqrt{n}L})$. From Theorem \ref{Th: rate} we have that $  \mE[\cG(\bar x_K)]  \leq  D/K$, where \ey{$D=(\tfrac{1}{\sigma}+2(1-\beta)(q+2)L)B_X=\mathcal O(\sqrt n L)$}, hence, 
we conclude that $ \mE[\cG(\bar x_K)]  \leq  \mathcal O(\tfrac{L\sqrt{n}}{K})$.
Moreover, to achieve $\mE[\cG(\bar x_K)]  \leq \epsilon$ the oracle (sample) complexity can be calculated as $n\lceil \frac{K}{q} \rceil+S (K-\lfloor \frac{K}{q} \rfloor)=\cO(n+\frac{nD}{q\epsilon}+\frac{S D}{\epsilon})$. Now by choosing $S = 1$ one can obtain $\cO(n+\frac{nD}{q\epsilon}+\frac{S D}{\epsilon}) = \cO(n+\sqrt{n}L/\epsilon)$. 
\end{proof}
\begin{remark}
We would like to emphasize that both SVR-APD \cite{yazdandoost2023stochastic} and VR-MP \cite{Alacaoglu2022Stochastic} achieve the same oracle complexity of $\mathcal{O}(n+\frac{\sqrt{n}}{\epsilon})$ using SVRG-type variance reduction. However, our method introduces a novel generalization of SPIDER-type variance reduction, specifically designed for VIs. Unlike SVR-APD, which addresses a specific minimax problem, our approach handles a more general case from the VI perspective. Moreover, while the VR-MP algorithm requires two computations of the Bregman proximal step in each iteration of the inner loop, our method requires only one. 
\end{remark}

\begin{remark}\label{beta = 0}[Simplfied updates for structured norm]
\eyy{It should be noted that we used some crude bounds when obtaining the bounds on the variance of sample operators in Lemma \ref{lemma: spider+} due to utilizing arbitrary normed vector space. These bounds can be improved when the vector space $\cX$ is equipped with a $p$-norm for $p\in \{1,2\}$. In particular, we can show that 
\begin{align*}
    &\mE\left[\fprod{\delta_x,w^x_k-w_{k+1}^x+x_k-x_{k+1}}\right]
\\
&\leq \tfrac{1}{\eta_k^3} \mE[\bD_X(w_{k+1}^x,w_k^x)]+\tfrac{1}{\eta_k^4}\mE[\bD_X(x_{k+1},x_k)]
 +2(1-\beta)^2L^2(\eta_k^3+\eta_k^4)\sum_{t=i_kq+1}^k\mE[\bD_X(x_t,x_{t-1})],
\end{align*}
which implies that the coefficient of consecutive iterates is no longer affected by the term $(1-\beta)q$ and one can select $\eta_k^1=\eta_k^2=\eta_k^3=\eta_k^4=\frac{1}{(1-\beta)L\sqrt{n}}$ leading to the stepsize $\sigma\leq \min\{\frac{1-\gamma}{2(1+\sqrt{n})L},\frac{\gamma}{\beta L (1+2\beta/\sqrt{n})}\}$. 
This implies that one can select $\beta=\gamma=0$ and $\sigma\leq \frac{1}{2(1+\sqrt{n})L}$ which simplifies the updates within our proposed algorithm while maintaining the complexity of $\cO(n+\frac{\sqrt{n}}{\epsilon})$. This insight underscores the challenge inherent in employing the Bregman distance function with an arbitrary norm in the algorithm design.}
\end{remark}


\subsection{Non-monotone VI}
In this section, we analyze our proposed method in a non-monotone setting. To enable this extension, we introduce essential assumptions and lemmas 
\eyh{for our analysis. In this scenario, we are interested in finding $\bar x\in X$ such that $\mathbf{0}\in F(\bar x)+\partial(h+\mathbb{I}_X)(\bar x)$. Consequently, we define the following gap function to find an $\epsilon$-approximated solution of \eqref{VI def}.
\begin{definition}
$\bar x\in X$ is an $\epsilon$-approximated solution when $\textbf{dist}(\mathbf{0},F(\bar x)+\partial (h+\mathbb{I}_X)(\bar x))\leq \epsilon$ where $\textbf{dist}(\cdot,\cdot)$ denotes the distance function, i.e., $\textbf{dist}(x,S)\triangleq \min_{s\in S} \|x-s\|_\cU$ for a given normed vector space $\cU$, a convex set $S\subset \cU$, and a point $x\in\cU$.
\end{definition}}%

\begin{assumption}[Weak Minty Solution]\label{assum:weak-minty}
There exists $x^* \in X$ and some $\rho >0$ such that 
\begin{align}
        - \rho\norm{v}_{\cX^*}^2 \leq \fprod{v,x-x^*} , \quad \forall x\in X,~ \forall v \in \partial (h+\mathbb{I}_X)(x)+F(x).
\end{align}
\end{assumption}
\zr{
   The recent studies on the weak Minty condition have shown that it holds under several structured conditions. In particular, this condition is satisfied when the operator $F$ is Lipschitz continuous and exhibits properties such as negative comonotonicity or positive cohypomonotonicity \cite{bohm2022solving, choudhury2023single}. 
   Furthermore, in the context of min-max optimization problem, the weak Minty condition (with $\rho=0$) contains all quasiconvex-concave and starconvex-concave problems \cite{gorbunov2022stochastic}. For further examples and applications see \cite{daskalakis2020independent,diakonikolas2021efficient,pethick2023escaping}.
}
\begin{assumption}\label{assum:bregman}
The Bregman distance generating function $\psi_\cX$ defined in Definition \ref{bregman def} has a Lipschtiz continuous gradient, i.e., for any $x,\bar x\in\cX$ we have $\norm{\grad\psi_\cX(x)-\grad\psi_\cX(\bar x)}_{\cX^*}\leq L_{\psi}\norm{x-\bar x}_\cX$ for some $L_\psi\geq 1$. 
\end{assumption}
An example of distance generating function satisfying  \eqref{assum:bregman} is the quadratic kernel $\psi_\cX(x)=\frac{1}{2}x^\top A x$ for some positive definite matrix $A$. One can also ensure this assumption by restricting the domain. For example, consider generalized Kullback–Leibler or Itakura–Saito divergence generated by the relative entropy function $\psi_\cX(x)=\sum_{i=1}^n x_i\log(x_i)$ and the logarithmic barrier distance generating function $\psi_\cX(x)=-\sum_{i=1}^n\log(x_i)$ within the domain $\reals^n_{++}$, respectively. The above assumption is satisfied when the domain is restricted to $\{x\in\reals^n_{+}\mid x\geq c\}$ for some $c> \mathbf{0}$. The application instance of such scenario include distributionally robust optimization where the robust empirical distributions set is selected as $\{x\in\reals^n_+\mid x\geq \frac{\epsilon}{n}\mathbf{1}_n,~\bD_X(x,\mathbf{1}_n/n)\leq \rho\}$ for some $\epsilon,\rho>0$ \cite{namkoong2016stochastic}. 

\ey{One of the main challenges in employing the Bregman distance function within the non-monotone regime is the difficulty in deriving a non-crude bound on the variance of the stochastic operator's error. This arises from the inability to leverage the Pythagorean theorem, a common tool when working with Euclidean distances. To address this, we establish a practical bound on the error of the stochastic operator, $\delta_k$, using only the triangle inequality. Notably, this bound can be further tightened when the norm is specifically chosen as the $l_2$-norm.} 
\begin{lemma}\label{lem:error-stochastic-nonmonotone}
    \ey{Let $\{x_k\}_{k\geq 0}$ be the sequence generated by Algorithm \ref{alg: VI-spider+}. Then, for any $k\geq 0$ and $a\in [0,1]$, $\|\delta_k\|_{\cX^*}^2\leq 4L^2(1-\beta)^a\sum_{t=i_kq+1}^k\bD_X(x_t,x_{t-1})$ such that $(1-\beta)^{2-a}q\leq 1$.} 
\end{lemma}
\begin{proof}
    \ey{
    Recall the definition of $\delta_k$ in Definition \ref{def param} and the relation in \eqref{define delta}. Using triangle inequality followed by Young's inequality with parameter $\eta>0$ implies that
    \begin{align*}
        \norm{\delta_k}^2_{\cX^*}&\leq (1-\beta)\left(\sum_{t=i_kq+1}^k\|e_t\|\right)\|\delta_k\|\\
        &\leq (1-\beta)\sum_{t=i_kq+1}^k\left(\tfrac{\eta}{2}\|e_t\|_{\cX^*}^2+\tfrac{1}{2\eta}\|\delta_k\|_{\cX^*}^2\right)\\
        &\leq \tfrac{\eta(1-\beta)}{2}\sum_{t=i_kq+1}^k\|e_t\|_{\cX^*}^2 ~ +\tfrac{(1-\beta)q}{2\eta}\|\delta_k\|_{\cX^*}^2,
    \end{align*}
    where in the last inequality we used $k-i_kq\leq q$ due to $i_k=\lfloor k/q \rfloor$. Next, rearranging the terms, using the upper bound $\|e_t\|$ in \eqref{bound for norm2 e^y}, and strong convexity of Bregman distance we obtain $\|\delta_k\|_{\cX^*}^2\leq \frac{4(1-\beta)\eta^2 L^2}{2\eta - (1-\beta)q}\sum_{t=i_kq+1}^k\bD_X(x_t,x_{t-1})$  if $2\eta-(1-\beta)q>0$. Moreover, selecting $\eta=\frac{1-\sqrt{1-(1-\beta)^{2-a}q}}{(1-\beta)^{1-a}}$ for any $a\in[0,1]$ such that $(1-\beta)^{2-a}q\leq 1$ implies that $\frac{4(1-\beta)\eta^2 L^2}{2\eta - (1-\beta)q}=4L^2(1-\beta)^a$. Finally, $2\eta-(1-\beta)q>0$ holds since $q\geq 0$ and $\beta\in (0,1)$.}
\end{proof}

\ey{Next, we show a one-step analysis of our proposed algorithm based on the above assumptions.}
\begin{lemma}\label{lem:non-monotne}
Let $\{x_k\}_{k\geq 0}$ be the sequence generated by VR-FoRMAB displayed in Algorithm \ref{alg: VI-spider+} initialized from arbitrary vectors $x_0 \in {\mathcal X}$. 
Suppose Assumptions \ref{lip def}, \ref{assum:weak-minty}, and \ref{assum:bregman} hold, and $ \za{r}_k$, $\delta_k$, and $\bar r_k$ are defined in Definition \ref{def param}. For any $k\geq 0$, the following result holds:
\begin{align}\label{eq:one-step-nonmonotone}
    0 &\leq \mE[\langle \bar r_{k+1} , x_{k+1} -x^* \rangle] -\theta_k\mE[\langle \bar r_k, x_{k} - x^* \rangle]   \\
     &\quad  +  \tfrac{1-\gamma_k} {\sigma_k} \mE[\mathbf{D}_X(x^*,x_k)-\mathbf{D}_X(x^*,x_{k+1})]+\tfrac{\gamma_k} {\sigma_k  } \mE[D^k_X(x^*)-\mathbf{D}_X(x^*,x_{k+1})] \nonumber\\
     &\quad +\mE\big[\cC^k_1\mathbf{D}_X(x_{k+1},x_k)+ \cC^k_2 D_X^k(x_{k+1}) +\cC^k_3 \mathbf{D}_X(x_k,x_{k-1}) + \cC^k_4 D_X^{k-1}(x_{k})\nonumber\\
     &\quad +  \cC^k_5 \sum_{t=\eyh{i_kq+1}}^k \bD_X(x_t,x_{t-1})\big],\nonumber
\end{align}
where $\cC^k_1\triangleq \ey{16\rho((1-\beta)^2L^2+ \tfrac{L_{\psi}^2(1-\gamma_k)^2}{\sigma_k^2})} +(\theta_k+q(1-\beta))L -\tfrac{1- \gamma_k} {\sigma_k}$ and $\cC^k_2\triangleq 16\rho(\beta^2 L^2+ \tfrac{L_{\psi}^2\gamma_k^2}{\sigma_k^2}) -\tfrac{\gamma_k} {\sigma_k}$, $\cC^k_3\triangleq 16\rho(1-\beta)^2 L^2\theta_k^2+\theta_k(1-\beta)L$, $\cC^k_4\triangleq 16\rho\beta^2 L^2\theta_k^2+\theta_k\beta L$, and $\cC^k_5\triangleq \ey{16\rho (1-\beta)^aL^2}+\ey{4(1-\beta) L}$. 
\end{lemma}
\begin{proof}
Considering the update of $x_{k+1}$, the optimality condition of the subproblem implies that
\begin{align*}
    0\in \partial (h+\mathbb{I}_X)(x_{k+1})+v_k+\theta_k \za{r}_k+\frac{1}{\sigma_k}(\grad \psi_\cX(x_{k+1})-\grad\psi_\cX(\hat x_k)),
\end{align*}
where $\mathbb{I}_X(\cdot)$ denotes the indicator function of set $X$. 
Adding $F(x_{k+1})$ to both sides of above inclusion, defining $\mathfrak{U}_k\triangleq F(x_{k+1}) - (v_k+\theta_k \za{r}_k+\frac{1}{\sigma_k}(\grad \psi_\cX(x_{k+1})-\grad\psi_\cX(\hat x_k)))$ and rearranging the terms implies that
\begin{align}\label{eq:opt-inclusion-x}
    \mathfrak{U}_k\in F(x_{k+1})+\partial (h+\mathbb{I}_X)(x_{k+1}).
\end{align}
Therefore, from Assumption \ref{assum:weak-minty} followed by applying Lemma \ref{lemma 5 similar} parts (a) and (d) we conclude that
\begin{align}\label{eq:opt-cond-bound}
    -\rho\norm{\mathfrak{U}_k}_{\cX^*}^2 &\leq \fprod{\mathfrak{U}_k,x_{k+1}-x^*}\\
    &= \fprod{v_k+\theta_k \za{r}_k - F(x_{k+1}), x^* - x_{k+1}} + \tfrac{1}{\sigma_k}\fprod{\grad \psi_\cX(x_{k+1})-\grad\psi_\cX(\hat x_k), x^*-x_{k+1}} \nonumber\\
    &= \langle  \delta_k , x^*-x_{k+1} \rangle + \theta_k\langle \za{r}_k - \bar r_k, x^*-x_{k} \rangle + \langle \bar r_{k+1} , x_{k+1} -x^* \rangle +\langle  \theta_k \bar r_k, x^*-x_{k} \rangle \nonumber \\
     &\quad + \langle \theta_k r_k, x_k-x_{k+1}\rangle +  \tfrac{1-\gamma_k} {\sigma_k} (\mathbf{D}_X(x,x_k)-\mathbf{D}_X(x,x_{k+1})-\mathbf{D}_X(x_{k+1},x_k))
     \nonumber\\
     &\quad+\tfrac{\gamma_k} {\sigma_k  } (D^k_X(x)-\mathbf{D}_X(x,x_{k+1})-D^k_X(x_{k+1}))\nonumber\\
     &= \langle  \delta_k , x^*-x_k \rangle + \theta_k\langle r_k - \bar r_k, x^*-x_{k} \rangle + \langle \bar r_{k+1} , x_{k+1} -x^* \rangle +\langle \theta_k \bar r_k, x^*-x_{k} \rangle \nonumber \\
     &\quad + \underbrace{\langle \delta_k + \theta_k r_k, x_k-x_{k+1} \rangle}_{(*)} +  \tfrac{1-\gamma_k} {\sigma_k} (\mathbf{D}_X(x^*,x_k)-\mathbf{D}_X(x^*,x_{k+1})-\mathbf{D}_X(x_{k+1},x_k))
     \nonumber\\
     &\quad+\tfrac{\gamma_k} {\sigma_k  } (D^k_X(x^*)-\mathbf{D}_X(x^*,x_{k+1})-D^k_X(x_{k+1})). \nonumber
\end{align}
Next, we provide an upper bound on $(*)$ using Lemma \ref{lemma: spider+} with $\eta_k^3=1/L$ as follows
\begin{align}\label{eq:first-inner-delta}
    (*) &= \mE[\langle \delta_k, x_k-x_{k+1} \rangle+\theta_k\fprod{r_k,x_k-x_{k+1}}] \\
    &\leq (1-\beta)qL \mE[\bD_X(x_{k+1},x_k)] + \ey{4(1-\beta)L} \sum_{t=i_kq+1}^k\mE[\bD_X(x_t,x_{t-1})] + \theta_k\mE[\fprod{r_k,x_k-x_{k+1}}]\nonumber.
\end{align}
\blue{The last term on the right-hand side of the above inequality can be bounded using Young's inequality. Note that $r_k$ is either $F_{S}(x_k)-(1-\beta)F_{S}({x}_{k-1})-\beta F_{S}(\tilde x_{k-1})$ or $\bar r_k$ which for both cases the inner product can be upper bounded similarly using Lipschitz continuity Assumption \ref{lip def} as follows:
\begin{equation}\label{eq:inner-rk}
\mE[\fprod{r_k,x_k-x_{k+1}}]\leq \frac{(1-\beta)L}{2}\mE[\|x_k-x_{k-1}\|_\cX^2]+\frac{\beta L}{2}\mE[\|x_k-\tilde x_{k-1}\|_\cX^2]+\frac{L}{2}\mE[\|x_{k+1}-x_k\|_\cX^2].
\end{equation}
Therefore, combining \eqref{eq:first-inner-delta} and \eqref{eq:inner-rk}, using strong convexity of Bregman distance function and Lemma \ref{lem:aux} we conclude that}
\begin{align}\label{eq:inner-prod-q-delta}
    (*)
    &\leq (\theta_k+{(1-\beta)q})L\mE[\bD_X(x_{k+1},x_k)] +\ey{4(1-\beta)L} \sum_{t=i_kq+1}^k\mE[\bD_X(x_t,x_{t-1})] \\
    &\quad +\theta_k L(1-\beta)\mE[\bD_X(x_k,x_{k-1})]+\theta_k\beta L\mE[D_{\zi{X}}^{k-1}(x_k)],\nonumber
\end{align}
\ey{where in the last inequality we used the definition of $r_k$,  Assumption \ref{lip def}, and Young's inequality similar to \eqref{inner pro bound q}.} 
On the other hand, from the definitions of $\delta_k$ and $\bar r_k$ in Definition \ref{def param} one can easily verify that $\mathfrak{U}_k=\bar r_{k+1}-\delta_k-\theta_kr_k-\frac{1}{\sigma_k}((\grad \psi_\cX(x_{k+1})-\grad\psi_\cX(\hat x_k)))$. First, we observe that from Lipschitz continuity of $\grad\psi_\cX$ and definitions of $\hat x_k$ and $s_k$ we obtain
\begin{align}\label{eq:grad-psi}
   &\norm{\grad \psi_\cX(x_{k+1})-\grad\psi_\cX(\hat x_k)}_{\cX^*}^2\\&\nonumber =  \norm{\grad \psi_\cX(x_{k+1})-(1-\gamma_k)\grad\psi_\cX(x_k)-\gamma_k s_k}_{\cX^*}^2  \\
    &\quad \leq 2(1-\gamma_k)^2\norm{\grad \psi_\cX(x_{k+1})-\grad\psi_\cX(x_k)}_{\cX^*}^2 +2\gamma_k^2\norm{\grad \psi_\cX(x_{k+1})-s_k}_{\cX^*}^2 \nonumber\\
    &\quad \leq \ey{4}(1-\gamma_k)^2L_\psi^2\bD_X(x_{k+1},x_k) +\ey{4}\gamma_k^2L_\psi^2 D_X^k(x_{k+1}). \nonumber
\end{align}
Moreover, from Assumption \ref{lip def}, the above inequality, Lemma \ref{lem:error-stochastic-nonmonotone}, and the fact that $(\sum_{i=1}^m a_i)^2\leq m\sum_{i=1}^m a_i^2$ for any $\{a_i\}_{i=1}^m\subset \reals_+$ we obtain
\begin{align}\label{eq:norm2-uk}
    \mE[\norm{\mathfrak{U}_k}_{\cX^*}^2]&\leq 4\mE[\norm{\bar r_{k+1}}_{\cX^*}^2] + 4\mE[\norm{\delta_k}_{\cX^*}^2] + 4\theta_k^2\mE[\norm{r_k}_{\cX^*}^2] + \tfrac{4}{\sigma_k^2}\mE[\norm{\grad \psi_\cX(x_{k+1})-\grad\psi_\cX(\hat x_k)}_{\cX^*}^2]  \\
    &\leq 16((1-\beta)^2L^2+ \tfrac{L_{\psi}^2(1-\gamma_k)^2}{\sigma_k^2}) \mE[\bD_X(x_{k+1},x_k)]+16(\beta^2 L^2+ \tfrac{L_{\psi}^2\gamma_k^2}{\sigma_k^2}) \mE[D^{k}_X(x_{k+1})]\nonumber\\
    &\quad + 16\ey{(1-\beta)^a} L^2 \sum_{t=\eyh{i_kq+1}}^k \mE[\bD_X(x_t,x_{t-1})] + 16(1-\beta)^2L^2\theta_k^2 \mE[\bD_X(x_k,x_{k-1})] \nonumber\\
    &\quad +16\beta^2 L^2\theta_k^2 \mE[D^{k-1}_X(x_k)].\nonumber
\end{align}


Now, taking expectation from inequality \eqref{eq:opt-cond-bound}, adding \eqref{eq:inner-prod-q-delta} as well as \eqref{eq:norm2-uk} multiplied by $\rho$, and rearranging the terms lead to the desired result.
\end{proof}

\eyh{Based on the one-step analysis derived above, we need to impose some conditions on the parameters to ensure a telescopic term in \eqref{eq:one-step-nonmonotone}. In the following lemma, we specify these parameters to ensure a telescopic structure in \eqref{eq:one-step-nonmonotone} based on the constants $C_i^k$ for $i\in\{1,\hdots,5\}$ and any $k\geq 0$. Later, in Theorem \ref{thm:rate-nonmontone} we used these conditions to derive a convergence rate guarantee based on \eqref{eq:one-step-nonmonotone}.}
\begin{lemma}\label{lem:nonmontone-stepsize}
    Let $\theta_k = 1$, $q=n$, \ey{$\beta\in (0,1)$ such that $(1-\beta)^{2-a} q\leq 1$ for some $a\in [0,1]$}, $\gamma_k = \gamma \in (0,1)$, and $\sigma_k = \sigma\leq \min\{\frac{\vartheta(1-\gamma)}{(3+\sqrt{7}\vartheta\zeta/L_\psi)L}, \frac{\vartheta\gamma}{(2 + \zeta'\vartheta/L_\psi)\beta L}\}=\cO(\frac{1}{L})$ for any $\vartheta\in(0,1)$ where $\zeta\triangleq 1/(1+4L_\psi/\vartheta)$ and $\zeta' \triangleq 1/(1+2L_\psi/(\sqrt{2}\vartheta))$. Assuming that $\rho\leq \frac{\vartheta\zeta'}{32\beta L L_\psi}$, then for any $k\geq 0$, 
    \begin{subequations}
    \begin{align}
        &\cC^k_1+\cC^k_3+q\cC^k_5+(1-\vartheta)\tfrac{1-\gamma_k}{\sigma_k}\leq 0, \label{eq:cond-sigma-1}\\
        &\cC^k_2+\cC^k_4+(1-\vartheta)\tfrac{\gamma_k}{\sigma_k}\leq 0,\label{eq:cond-sigma-2}
    \end{align}
    \end{subequations}
    where $\cC^k_i$'s are defined in the statement of Lemma \ref{lem:non-monotne}.
\end{lemma}
\begin{proof}
Consider the desired inequality \eqref{eq:cond-sigma-1} and let $\sigma_k=\sigma$ and $\theta_k=1$. Multiplying both sides by $\sigma^2$ and defining \ey{$\chi_1\triangleq 2(1-\beta)^2+q(1-\beta)^a$ and $\chi_2\triangleq 2-\beta+5q(1-\beta)$} imply that \eqref{eq:cond-sigma-1} holds if and only if $a_1 \sigma^2-a_2\sigma+a_3\leq 0$, where \ey{$a_1\triangleq 16\chi_1\rho L^2+\chi_2L$}, $a_2\triangleq \vartheta(1-\gamma)$, and $a_3\triangleq 8\rho L_\psi^2(1-\gamma)^2$. Therefore, assuming that $a_2^2-4a_1a_3\geq 0$, $\sigma$ must satisfy $\sigma\leq \frac{a_2+\sqrt{a_2^2-4a_1a_3}}{2a_1}$ which holds if $\sigma\leq \frac{a_2}{2a_1}$.
Note that $a_2^2-4a_1a_3\geq 0$ can be rewritten as a quadratic inequality in terms of $\rho$ as follows
\begin{equation}\label{eq:cond-rho}
    b_1\rho^2+b_2\rho-\vartheta^2\leq 0,
\end{equation}
where \ey{$b_1\triangleq 512 \chi_1 L^2L_\psi^2$, and $b_2\triangleq 32 \chi_2 L L_\psi^2$}. 
From \eqref{eq:cond-rho} we have that $\rho\leq \frac{-b_2+\sqrt{b_2^2+4\vartheta^2 b_1}}{2b_1}$ and a simple algebra reveals $\sqrt{b_2^2+4b_1\vartheta^2}$ $\geq b_2+2\zeta\vartheta \sqrt{b_1}$, for any \ey{$\zeta\leq 1/(1+2L_\psi\chi_2/(\vartheta \sqrt{\chi_1}))$}. Therefore, one can immediately conclude that $\rho\leq \frac{\vartheta\zeta}{\sqrt{b_1}}$
\ey{ and together with $\frac{a_2}{2a_1}=\ey{\frac{16\vartheta(1-\gamma)L_\psi^2}{b_1\rho + b_2}}$ we conclude that $\sigma\leq \frac{\vartheta(1-\gamma)}{(\chi_2+\sqrt{\chi_1}\vartheta\zeta/(2L_\psi))L}$ implies that $\sigma\leq\frac{a_2}{2a_1}$}. \ey{Moreover, when $n$ is large enough, one can verify that $\beta$ can be selected such that $\chi_1\in [1,3],\chi_2\in [6,7]$.}

Next, considering inequality \eqref{eq:cond-sigma-2} one can follow the same line of proof to show the result. In particular, \eqref{eq:cond-sigma-2} leads to the following quadratic inequality $c_1\sigma^2-c_2\sigma+c_3\leq 0$ where $c_1\triangleq 16\rho\beta^2L^2+\beta L$, $c_2\triangleq \gamma\vartheta$, and $c_3\triangleq \ey{16}\rho L_\psi^2\gamma^2$. Assuming that $d_1\rho^2+d_2\rho-\vartheta^2\leq 0$ where $d_1\triangleq 1024 \beta^2L^2L_\psi^2$ and $d_2\triangleq 64\beta L L_\psi^2$, we conclude that $\sigma$ satisfies $\sigma \leq \frac{c_2}{2c_1}=\frac{\gamma\vartheta}{32\rho\beta^2L^2+2\beta L}$. 
Moreover, $\rho\leq \frac{-d_2+\sqrt{d_2^2+4\vartheta^2d_1}}{2d_1}$ holds when $\rho\leq \frac{\zeta'\vartheta}{\sqrt{d_1}}$ for any $\zeta'\leq 1/(1+\sqrt{2}L_\psi/\vartheta)$, hence, \eqref{eq:cond-sigma-2} holds if $\sigma\leq \frac{\vartheta\gamma}{(2 + \zeta'\vartheta/L_\psi)\beta L}$. Finally, comparing the two upper-bounds obtained for $\rho$ by noting that $\beta\in (0,1)$ we have that $\zeta/\sqrt{b_1}\leq \zeta'/\sqrt{d_1}$, hence, $\rho\leq \frac{\vartheta\zeta'}{\sqrt{d_1}}\leq \frac{\vartheta\zeta}{\sqrt{b_1}}$. 
\end{proof}

\begin{theorem}\label{thm:rate-nonmontone}
Let $\{x_k\}_{k\geq 0}$ be the sequence generated by VR-FoRMAB displayed in Algorithm \ref{alg: VI-spider+} initialized
from arbitrary vector $x_0 \in {\mathcal X}$. For any $k\geq 0$ let $\mathfrak{U}_k\triangleq F(x_{k+1}) - (v_k+\theta_k\za{r}_k+\frac{1}{\sigma_k}(\grad \psi_\cX(x_{k+1})-\grad\psi_\cX(\hat x_k)))$. Suppose Assumptions \ref{lip def}, \ref{assum:weak-minty}, and \ref{assum:bregman} hold and the step-size sequence $\sigma_k$ and parameters $\gamma_k,\beta$, and $q$ are selected as in the statement of Lemma \ref{lem:nonmontone-stepsize} for any $k\geq 0$. Then, for any $K> 1$, 
\begin{align}\label{eq:rate-nonmonotone}
    \sum_{k=0}^{K-2}\mE[\norm{\mathfrak{U}_k}_{\cX^*}^2]\leq \tfrac{1}{1-\vartheta}\max\{\cE_1,\cE_2\} \mathbf{D}_X(x^*,x_0)=\ey{\cO\left((q(1-\beta)^a+ q^2(1-\beta)^2)L^2\right)},
\end{align}
where $\cE_1\triangleq 16\ey{(2(1-\beta)^2+q(1-\beta)^a}L^2)/(1-\gamma)+16{L_\psi^2(1-\gamma)}/{\sigma^2}$ and $\cE_2\triangleq 32\beta^2L^2/\gamma +16 {L_\psi^2\gamma}/{\sigma^2}$.
\end{theorem}
\begin{proof}
Since the step size and algorithm's parameters satisfy the conditions of Lemma \ref{lem:nonmontone-stepsize} we have that \eqref{eq:cond-sigma-1} and \eqref{eq:cond-sigma-2} hold. Therefore, using Lemma \ref{lem:aux} and thanks to the telescopic summation we have
\begin{align}\label{eq:sum-step-size}
    &\sum_{k=0}^{K-1} \mE\big[\cC^k_1\mathbf{D}_X(x_{k+1},x_k)+ \cC^k_2 D_X^k(x_{k+1}) +\cC^k_3 \mathbf{D}_X(x_k,x_{k-1}) + \cC^k_4 D_X^{k-1}(x_{k})\\
     &\quad +  \cC^k_5 \sum_{t=\eyh{i_kq+1}}^k \bD_X(x_t,x_{t-1})\big]\nonumber\\
     &\leq -(1-\vartheta)\sum_{k=0}^{K-1}\mE[\tfrac{1-\gamma}{\sigma}\mathbf{D}_X(x_{k+1},x_k) + \tfrac{\gamma}{\sigma}D_X^k(x_{k+1})]. \nonumber
\end{align}
Next, summing the result of Lemma \ref{lem:non-monotne} over $k$ from $0$ to $K-1$ and using \eqref{eq:sum-step-size} we conclude that
\begin{align}\label{eq:sum}
    &(1-\vartheta)\sum_{k=0}^{K-1}\mE[\tfrac{1-\gamma}{\sigma}\mathbf{D}_X(x_{k+1},x_k) + \tfrac{\gamma}{\sigma}D_X^k(x_{k+1})]\\
    &\leq \mE[\langle \bar r_{K} , x_{K} -x^* \rangle] -\mE[\langle \bar r_0, x_{0} - x^* \rangle] +  \tfrac{1-\gamma} {\sigma} \mE[\mathbf{D}_X(x^*,x_0)-\mathbf{D}_X(x^*,x_{K})] \nonumber  \\
     &\quad  +\tfrac{\gamma} {\sigma} \mE[D^{K-1}_X(x^*)-\mathbf{D}_X(x^*,x_{K})], \nonumber \\
     &\leq \mE[{(1-\beta)L}\bD_X(x_K,x_{K-1})+\beta L D_X^{K-1}(x_K) + L\bD_X(x^*,x_K)] \nonumber  \\
     &\quad  +  \tfrac{1-\gamma} {\sigma} \mE[\mathbf{D}_X(x^*,x_0)-\mathbf{D}_X(x^*,x_{K})]+\tfrac{\gamma} {\sigma} \mE[D^{K-1}_X(x^*)-\mathbf{D}_X(x^*,x_{K})], \nonumber
\end{align}
where in the last inequality we followed a similar step as in \eqref{inner pro bound q} for $x=x^*$ to provide an upper bound on $\langle \bar r_K , x_K -x^* \rangle \za{-} \langle\bar r_0, x_0-x^* \rangle $ by using Young's inequality and noting that $\bar r_0=0$. Since $\sigma$ satisfies $\sigma\leq \min\{\frac{1-\gamma}{2L},\frac{\gamma}{2L}\}$ we can simplify the terms in \eqref{eq:sum} to obtain
\begin{align*}
    (1-\vartheta)\sum_{k=1}^{K-1}\mE[\tfrac{1-\gamma}{\sigma}\mathbf{D}_X(x_{k},x_{k-1}) + \tfrac{\gamma}{\sigma}D_X^{k-1}(x_{k})]\leq \tfrac{1} {\sigma} \mathbf{D}_X(x^*,x_0) .
\end{align*}
Multiplying both sides by $\sigma/(1-\vartheta)$ leads to
\begin{align}\label{eq:rate-consecutive}
    \sum_{k=1}^{K-1}\mE[ (1-
    \gamma) \mathbf{D}_X(x_{k},x_{k-1}) + \gamma D_X^{k-1}(x_{k})]\leq \tfrac{1}{1-\vartheta}\mathbf{D}_X(x^*,x_0) .
\end{align}

Next, recalling the upper-bound on $\mE[\norm{\mathfrak{U}_k}^2_{\cX^*}]$ in \eqref{eq:norm2-uk} in terms of the consecutive iterates and summing over $k=0$ to $K-2$ implies that
\begin{align*}
    &\sum_{k=0}^{K-2}\mE[\norm{\mathfrak{U}_k}_{\cX^*}^2]\leq\\
    & \quad \ey{16((1-\beta)^2L^2}+ \tfrac{L_{\psi}^2(1-\gamma)^2}{\sigma^2}) \sum_{k=0}^{K-2} \mE[\bD_X(x_{k+1},x_k)]+16(\beta^2 L^2+ \tfrac{L_{\psi}^2\gamma^2}{\sigma^2}) \sum_{k=0}^{K-2} \mE[D^{k}_X(x_{k+1})]\nonumber\\
    &+ 16(\ey{(1-\beta)^aq} + (1-\beta)^2)L^2 \sum_{k=0}^{K-2}\mE[\bD_X(x_k,x_{k-1})]  +16\beta^2 L^2 \sum_{k=0}^{K-2} \mE[D^{k-1}_X(x_k)].
\end{align*}
Now, using the definitions of $\cE_1$ and $\cE_2$ from the statement of the theorem in the last inequality and \ey{the facts that $x_{-1}=x_0$ and $(1-\beta)q\leq 1$}, we have that
\begin{align*}
   &\sum_{k=0}^{K-2}\mE[\norm{\mathfrak{U}_k}_{\cX^*}^2] \leq \sum_{k=1}^{K-1} \big(\cE_1 (1-\gamma) \mE[\bD_X(x_{k},x_{k-1})] + \cE_2 \gamma \mE[D^{k-1}_X(x_{k})]\big). \nonumber
\end{align*}
Finally, combining the above inequality with \eqref{eq:rate-consecutive} leads to the desired result.
\end{proof}

\begin{corollary}
Under the premises of Theorem \ref{thm:rate-nonmontone} by choosing \ey{$a=1$, $\gamma\in(0,1)$, and $\beta=1-\frac{1}{n}$} there exists $t\in\{1,\hdots,K-1\}$ such that 
\begin{align}
\mE[\textbf{dist}(\mathbf{0},F(x_{t})+\partial (h+\mathbb{I}_X)(x_{t}))]\leq \epsilon,
\end{align}
within $K=\cO(L^2/\epsilon^2)$ iterations. Moreover, the oracle
complexity to achieve an $\epsilon$-approximated solution is $\cO(n+L^2/\epsilon^2)$.
\end{corollary}
\begin{proof}
From Theorem \ref{thm:rate-nonmontone} and defining $t=\argmin_{0\leq k\leq K-2}\{\norm{\mathfrak{U}_k}_{\cX^*}\}$ we conclude that for any $K>1$, $\mE[\norm{\mathfrak{U}_t}^2_{\cX^*}]\leq \frac{\max\{\cE_1,\cE_2\}}{(1-\vartheta)K-1} \mathbf{D}_X(x^*,x_0)$. Now, let $A(x_{k+1})\triangleq F(x_{k+1})+\partial (h+\mathbb{I}_X)(x_{k+1})$ and by noting that $\mathfrak{U}_{k}\in A(x_{k+1})$ for any $k\geq 0$ we have 
\begin{equation}
    \mE[\norm{\mathfrak{U}_t}^2_{\cX^*}]\geq (\mE[\norm{\mathfrak{U}_t}_{\cX^*}])^2\geq  \Big(\mE\big[\min_{u\in A(x_{t+1})}\norm{u}_{\cX^*}\big]\Big)^2=\big(\mE[\textbf{dist}(\mathbf{0},A(x_{t+1}))]\big)^2.
\end{equation}
Hence, one can readily conclude that $\mE[\textbf{dist}(\mathbf{0},A(x_{t+1}))]\leq \sqrt{\frac{\max\{\cE_1,\cE_2\}}{(1-\vartheta)K-1} \mathbf{D}_X(x^*,x_0)}$ which in light of selection of parameters as in the statement we have that $\max\{\cE_1,\cE_2\}=\cO(L^2)$ leading to the desired result. 
\end{proof}

\begin{remark}\label{rem:stepsize}
We can simplify the parameter selection of our proposed method by choosing \ey{$\beta = 1-\frac{1}{n}$, $\gamma \in (0,1)$, and $\sigma \leq \frac{(1-\gamma)}{6L}$. It is worth noting that, unlike the monotone setting the parameters $\beta$ and $\gamma$ do not vanish when the Bregman distance is specified to Euclidean distance.} This indeed shows the impact of novel momentum terms and stochastic operator approximation in our method to ensure the obtained complexity result.  
\end{remark}
\begin{remark}
The condition on the weak-MVI parameter $\rho$ in our analysis is $\rho<\frac{1}{32 LL_\psi(1+\sqrt{2}L_\psi)}$ which is slightly worse than some of the existing results \cite{pethick2023solving}. However, it is important to note that such a loose bound stems from our consideration of an arbitrary norm in our framework and using some crude bounds such as those in \eqref{eq:norm2-uk}. It is reasonable to expect that specializing our analysis for Euclidean space could lead to an improvement of the constant in the upper bound of $\rho$. Nevertheless, our result encompasses a broader range of problems, and notably, we have achieved an improved convergence rate compared to existing methods \cite{pethick2023escaping,pethick2023solving} which has appeared for the first time in the literature.
\end{remark}

\section{Experimental results}
\label{sec:experiments}
In this section, we implement VR-FoRMAB for solving DRO example in section \ref{intro} and a non-monotone two-player matrix game. We compare our method with other state-of-the-art methods. All experiments are performed on a machine running 64-bit Windows 11
with Intel i5-1135G7 @2.40GHz and 8GB RAM. 

\subsection{Distributionally Robust Optimization} We implement our 
method to solve the DRO problem \eqref{DRO relaxation} described in Section \ref{intro} and compare it with Stochastic Variance Reduced Accelerated Primal-Dual Method (SVR-APD) \cite{yazdandoost2023stochastic} and Mirror-Prox with Variance Reduction (VR-MP)  \cite{Alacaoglu2022Stochastic}. Following the setup in \cite{yazdandoost2023stochastic}, we consider a set of labeled data points consisting of feature vectors $\{a_i\}_{i=1}^{n} \subset \mathbb{R}^d$ and labels $\{b_i\}_{i=1}^{n} \subseteq \{+1,-1\}^{n} $. Let $\ell_i(x) = \log(1 + \text{exp}(-b_ia_i^\top x))$ and $V(y,\tfrac{1}{n}\mathbf 1_n) = \tfrac{1}{2}\|ny-\mathbf 1_n\|^2$ serves as the Chi-square divergence measure. Furthermore, we specify $X = [-10, 10]^d$ and $\rho = 50$. The comparison of methods is conducted across different datasets, and a summary of the datasets is presented in Table \ref{tab:my_table}. \eyh{For all methods, projection onto the simplex-set constraint in the maximization is computed} by utilizing the Bregman distance with generating function $ \psi_{\cY}(y) = \sum_{i=1}^{n} y_i \log(y_i)$. 
\eyh{We selected the VR-FoRMAB's parameters according to Remark \ref{beta = 0}, i.e., $\sigma =\frac{1-\beta}{(2\sqrt{n}+4)L}$, $\gamma= \beta = 0$, $S =1$, $q =n$. Similarly, the parameters of SVR-APD and VR-MP are selected based on their theoretical suggestion.} 
Let us define the Lagrangian function $\mathcal{L}(x,y) \triangleq f(x) +\sum_{i=1}^n L_i(x,y) - h(y)  $ and recall that $L_i(x,y) = ny_i\ell_i(x)-\lambda(\frac{1}{2}V_i(y_i,\tfrac{1}{n}\mathbf 1_n)-\tfrac{\rho}{n}))$, $f(x)= \mathbb I_{X \times \mathbb R_{+}}(u,\lambda)$and $h(y) = \mathbb I_{\Delta_n}(y) $. The results are depicted in Figure \ref{fig:DRO Num}, where we plot the results in terms of the difference of Lagrangian functions evaluated at a solution pair $(x^*,y^*)$, i.e. $\mathcal{L}(x_k,y^*)-\mathcal{L}(x^*,y_k)$\footnote{Note that $0\leq \mathcal{L}(x_k,y^*)-\mathcal{L}(x^*,y_k)\leq \sup_{z\in X}\fprod{F(z^*),z-z^*}=\cG(z_k)$.}, versus time. Moreover, in Table \ref{gap at last iter} we compare the gap function defined in \eqref{gap function} at the last iterate point for all algorithms. Our method demonstrates better performance compared to the other two approaches across all experiments. 
\begin{table}[!h]
\centering
\caption{Different datasets used in the experiment from LIBSVM \cite{chang2011libsvm}.}
{\begin{tabular}{|c|c|c|c|c|}
\hline
& \texttt{w2a} & \texttt{Mushroom}  & \texttt{Phishing} & \texttt{a7a}\\
\hline
samples & 3470 & 8124 &11055 &16100    \\
\hline
features & 300 & 112 &  64 &122 \\
\hline
\end{tabular}}

\label{tab:my_table}
\end{table}

\vspace{-20pt} 

\begin{figure}[htbp]
    \centering
    \includegraphics[scale=0.23]{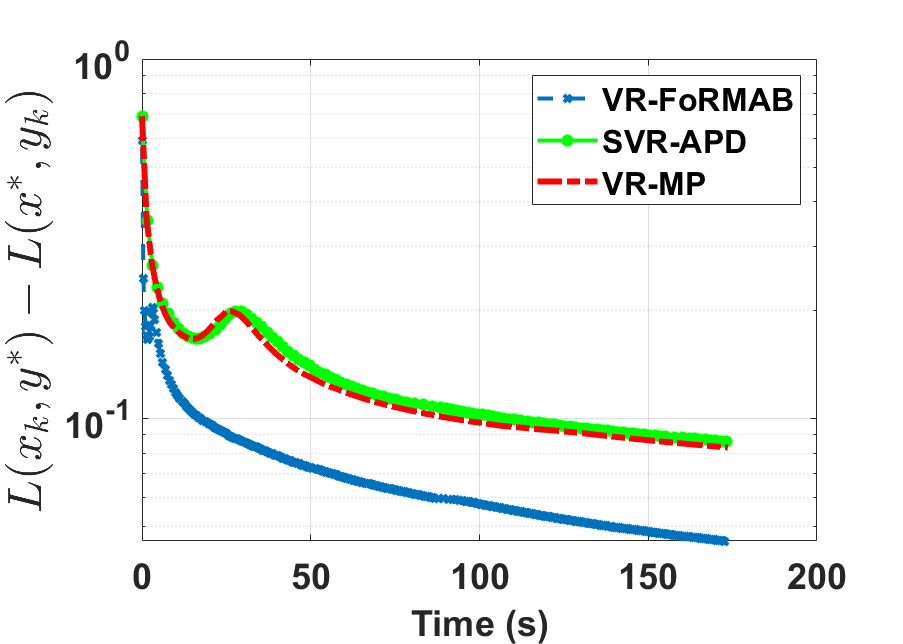}
  \includegraphics[scale=0.23]{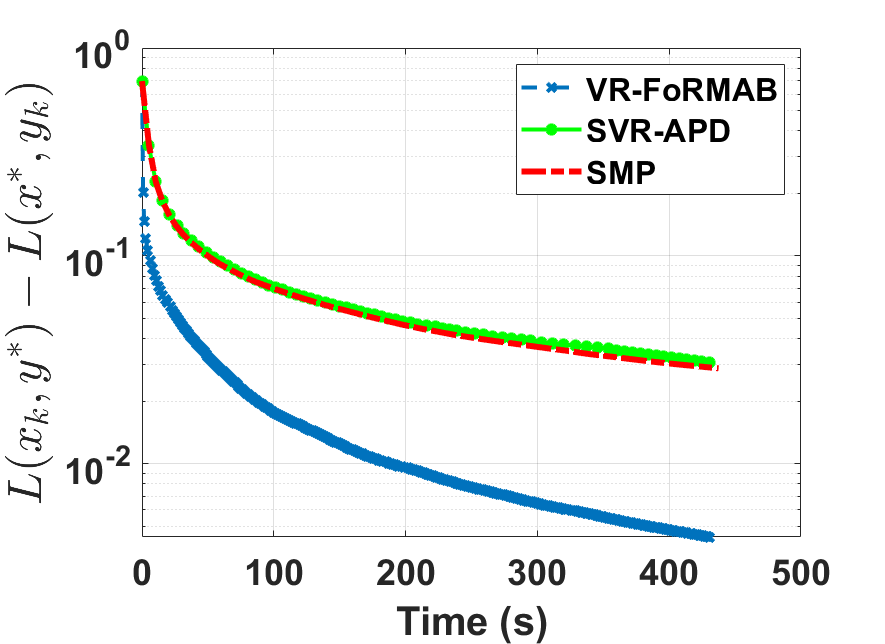}
      \includegraphics[scale=0.23]{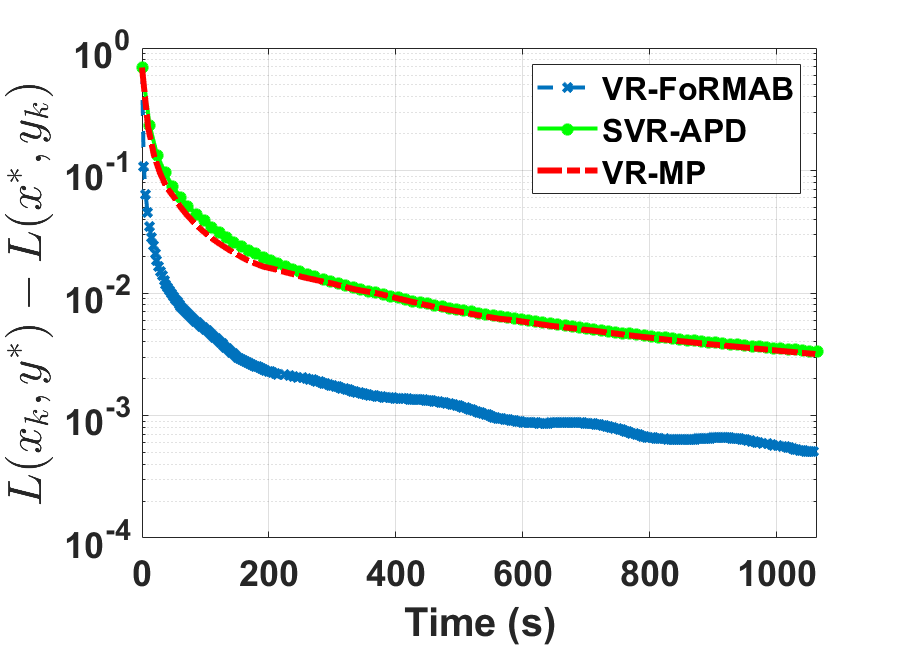}
    \includegraphics[scale=0.23]{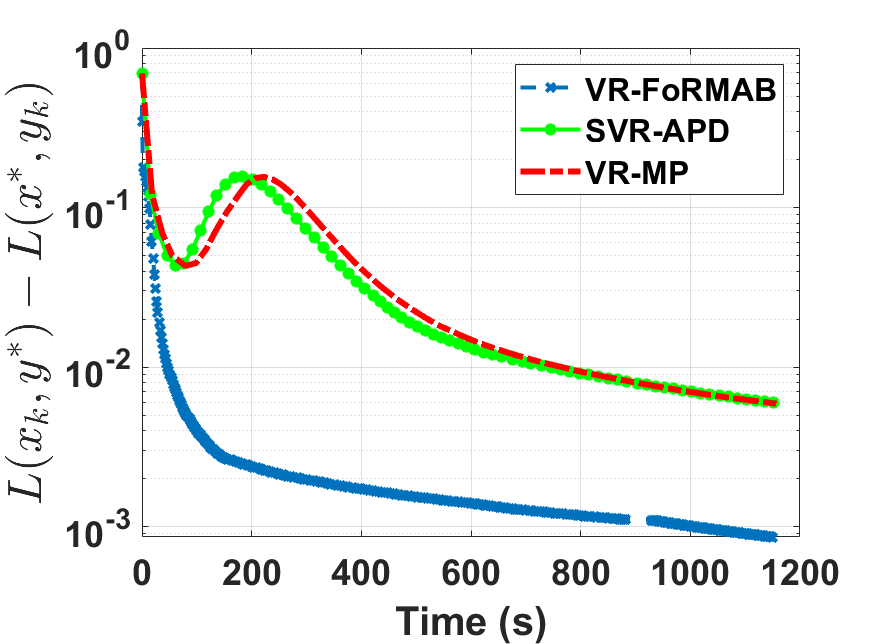}
    \caption{Comparison of the methods in terms of time for different datasets: Top row from left to right \texttt{w2a} and \texttt{Mushrooms}, Bottom row left to right \texttt{Phishing} and \texttt{a7a}.}
    \label{fig:DRO Num}
\end{figure}
\begin{table}[h!]
\centering
\caption{Comparing gap function $\sup_{(x,y)\in\cX\times\cY}\{\mathcal{L}(x_K,y)-\mathcal{L}(x,y_K) \}$ at the last iterate for different datasets.}
{\begin{tabular}{|c|c|c|c|c|}
\hline
                  & \texttt{w2a} & \texttt{Mushroom}  & \texttt{Phishing}& \texttt{a7a}\\ \hline
VR-FoRMAB                 &   \bf 5.1e-2     &   \bf 4.7e-3    &   \bf 1.2e-2               &            {\bf 1.1e-3}       \\ \hline
 SVRAPD             &     1.3e-1        &     2.6e-2              &     9.6e-2               &       2.3e-2            \\ \hline
VR-MP  & 1.2e-1   &1.9e-2 &  8.6e-2& 2.3e-2 \\ \hline
\end{tabular}}

\label{gap at last iter}
\end{table}
\subsection{Matrix Game} In this section, we compare the performance of our proposed method in a non-monotone setting with the state-of-the-art algorithms for solving a two-player matrix game formulated as follows:
\begin{equation}\label{eq:matrix-game}
    \min_{\norm{u}_2\leq 1}\max_{\norm{w}_2\leq 1} -\frac{\upsilon}{2}\norm{u}^2_2+\fprod{Au,w}+\frac{\upsilon}{2}\norm{w}^2_2.
\end{equation}
A first-order stationary solution of the above problem can be found by solving the corresponding VI formulated as \eqref{VI def} by letting $x=\begin{bmatrix} u\\ w\end{bmatrix}$, $F(x)=\begin{bmatrix} A^\top w-\upsilon u\\ -Au-\upsilon w\end{bmatrix}$, $h\equiv 0$, and $X=\{u\in\reals^n\mid \norm{u}_2\leq 1\}\times\{w\in\reals^m\mid \norm{w}_2\leq 1\}$. 
\blue{It is easy to verify that \( F \) is non-monotone, since the symmetric part of its Jacobian is indefinite when \( A \) is non-symmetric. Also, similar to  ~\cite{pethick2023escaping}, we can show \( F \) is Lipschitz continuous with constant \( L = \sqrt{\upsilon^2 + \|A\|_2^2} \), and $F$ satisfies Assumption \ref{assum:weak-minty} with constant $\rho=\frac{\upsilon}{\upsilon^2+\norm{A}_2^2}$.}  In our experiment, we let $\upsilon=1$, generate $A$ randomly from a standard Gaussian distribution then normalized such that $\norm{A}_2=40$, $n=m$, and $F_i(x)=\begin{bmatrix} nA_{i:} w_i-\upsilon u\\ -mA_{:i} u_i-\upsilon w\end{bmatrix}$. We compare our proposed method with VR-MP and BC-SEG+ \cite{pethick2023solving}. The parameters of our method are selected based on Remark \ref{rem:stepsize}, i.e., $\gamma=\beta=1-\frac{1}{2n+1}$ and $\sigma=\frac{\nu}{6L}$ for some $\nu>0$. For VR-MP we select the step size as $\sigma = \frac{\nu}{L}$ and for BC-SEG+ we used $\alpha_k=\nu/k$. In our experiment, we select $\nu\in\{0.1,1,10\}$ as a tuning parameter to observe the performance of the considered methods with different step-sizes. We tested the performance of the considered methods for various dimensions $n=m\in\{100,500,1000\}$ and the plots are depicted in Figure \ref{fig:non-monotone}. Since $x=\mathbf{0}$ is the equilibrium of this problem, the plots are in terms of the normalized solution norm, i.e., $\norm{x_k}/\sqrt{n}$, versus the number of sample operators evaluated by the algorithms. Conversely, BC-SEG+ demonstrates convergence in lower dimensions but exhibits sluggish convergence when $n=1000$, possibly due to increased variance in sampling and a diminishing step-size. Our proposed method, VR-FoRMAB, consistently converges to the solution across all experiments, surpassing other methods.

\begin{figure}
    \centering
    \includegraphics[scale=0.22]{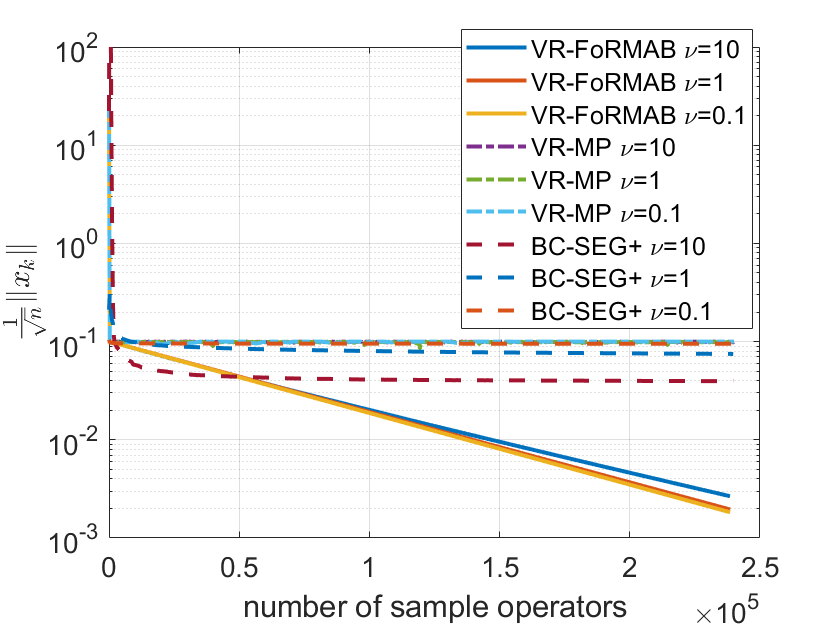} \hspace{-3mm} 
    \includegraphics[scale=0.22]{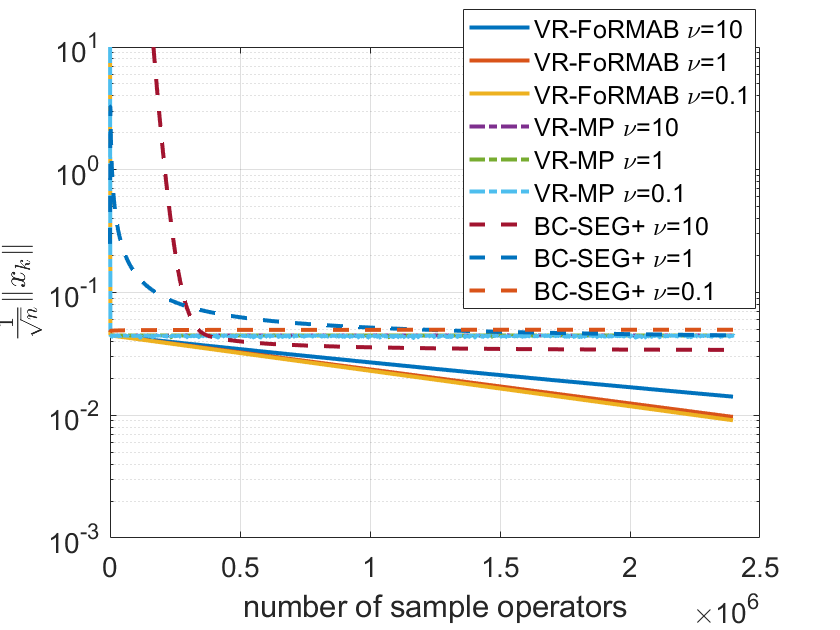} \hspace{-3mm} 
    \includegraphics[scale=0.22]{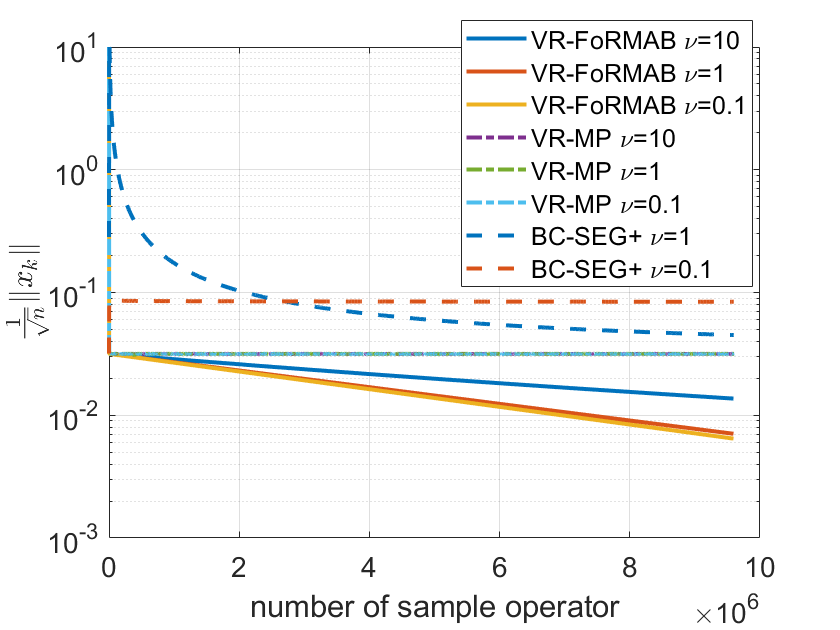}
    \caption{Comparison of methods for solving \eqref{eq:matrix-game} for various dimensions. From left to right, $n=100$, 500, and 1000. }
    \label{fig:non-monotone}
\end{figure}

\section{Conclusions}
\label{sec:conclusions}
In this paper, we studied VI problems with a finite-sum structure. Motivated by the lack of a unified method capable of adapting a Bregman distance function to address both monotone and non-monotone settings, we introduced a novel stochastic method called Variance-Reduced Forward Reflected Moving Average Backward (VR-FoRMAB). We established the convergence rate and oracle complexity of our method for the monotone case, achieving the optimal complexity of $\cO(n+\frac{\sqrt{n}}{\epsilon})$.
Additionally, we analyzed the convergence of our algorithm in the non-monotone setting under the weak Minty Variational Inequality (MVI) assumption and the Lipschitz continuity of the Bregman distance generating function. Our proposed method attains a complexity of $\mathcal{O}(n+\frac{1}{\epsilon^2})$, offering an improvement in oracle complexity over existing approaches.


\section*{Funding}
This work was supported by the National Science Foundation (NSF) under Grant ECCS-2231863. 

\section*{Disclosure statement}
The authors report there are no competing interests to declare.






\bibliographystyle{abbrv}
\bibliography{ref}


\end{document}